\documentclass[11pt]{amsart}

\usepackage{amsmath}
\usepackage{amssymb,mathtools,amsthm,latexsym}
\usepackage{mathrsfs}
\usepackage{bm}
\usepackage{url}
\usepackage{xcolor}
\numberwithin{equation}{section}

\usepackage[nocompress]{cite}

\calclayout

\usepackage{comment}
\usepackage{soul}

\newtheorem{thm}{Theorem}[section]
\newtheorem{lemma}[thm]{Lemma}
\newtheorem{cor}[thm]{Corollary}
\theoremstyle{definition}
\newtheorem{defn}[thm]{Definition}

\newtheorem{remark}[thm]{Remark}
\newtheorem*{remark*}{Remark}
\newtheorem{prop}[thm]{Proposition}

\def\Xint#1{\mathchoice
	{\XXint\displaystyle\textstyle{#1}}%
	{\XXint\textstyle\scriptstyle{#1}}%
	{\XXint\scriptstyle\scriptscriptstyle{#1}}%
	{\XXint\scriptscriptstyle\scriptscriptstyle{#1}}%
	\!\int}
\def\XXint#1#2#3{{\setbox0=\hbox{$#1{#2#3}{\int}$}
		\vcenter{\hbox{$#2#3$}}\kern-.5\wd0}}

\def\dashint{\Xint-}

\newcommand{\R}{\ensuremath{\mathbb{R}}}
\newcommand{\Q}{\ensuremath{\mathbb{Q}}}

\newcommand{\LIP}{\ensuremath{\mathrm{LIP}}}

\newcommand{\defeq}{\mathrel{\mathop:}=}

\newcommand{\eps}{\varepsilon}

\newcommand{\ud}{\mathrm{d}}

\newcommand{\inv}{^{-1}}

\newcommand{\spt}{\mathrm{spt}}
\newcommand{\id}{\mathrm{id}}
\newcommand{\im}{\mathrm{im}}

	\title[Generalized products and Lorentzian length spaces]{Generalized products and Lorentzian length spaces}
\author{Elefterios Soultanis}

\address{Department of Mathematics and Statistics\\
	University of Jyv\"askyl\"a\\
	P.O. Box 35\\
	FI-40014 University of Jyv\"askyl\"a\\
	Finland}
\email{\tt elefterios.soultanis@gmail.com}

\begin{document}

\begin{abstract}
	We construct a Lorentzian length space with an orthogonal splitting on a product $I\times X$ of an interval and a metric space, and use this framework to consider the relationship between metric and causal geometry, as well as synthetic time-like Ricci curvature bounds. 
	
	The generalized Lorentzian product naturally has a \emph{Lorentzian length structure} but can fail the push-up condition in general. We recover the push-up property under a log-Lipschitz condition on the time variable and establish sufficient conditions for global hyperbolicity. Moreover we formulate time-like Ricci curvature bounds without push-up and regularity assumptions, and obtain a partial rigidity of the splitting under a strong energy condition. 
\end{abstract}

\maketitle

\section{Introduction}

\subsection{Background and outline of results}

A (smooth) globally hyperbolic metric $g$ on a spacetime $M$ admits an orthogonal splitting 
\begin{align}\label{eq:orth_split}
	g=-h^2\ud t^2+g_t,\quad M=(a,b)\times S
\end{align}
where $g_t$ is a ``time dependent'' Riemannian metric on the spacelike hypersurface $\{t\}\times S\subset M$ and $h\colon M\to\R$ is a positive (smooth) function \cite{ber-san05}. Galloway's splitting theorem \cite{gal89} further improves \eqref{eq:orth_split} to the isometric identification  $(M,g)\simeq(\R\times S,-\ud t^2+g_S)$ (with $g_S$ constant in $t$) assuming the \emph{strong energy condition} ${\rm Ric}(X,X)\ge 0$ for time-like vectors $X$, and that $M$ contains a time-like geodesic line. Galloway's result can be interpreted as \emph{rigidity} imposed on \eqref{eq:orth_split} by the strong energy condition and the existence of complete time-like lines.

For continuous globally hyperbolic metrics only a \emph{topological} splitting $M\approx(a,b)\times S$ is available, as a result of the stability of global hyperbolicity \cite{sae16}, see also \cite{ger70}. The continuous case is an important step towards a theory of non-smooth Lorentzian geometry, motivated e.g. by the study of singularity theorems  and by Penrose's cosmic censorship conjecture. In this direction, the so called \emph{Lorentzian pre-length spaces}, introduced by Kunzinger--S\"amann \cite{kun-sam18}, have become a popular framework. They consist of a causal set $Y=(Y,\le,\ll)$ satisfying the \emph{push-up property}, and a time separation function $\tau:Y\times Y\to [0,\infty]$ satisfying a \emph{reverse triangle inequality} (cf. Definition \ref{def:lor-pre-length}). Lorentzian pre-length spaces provide an axiomatization of the causal geometry of smooth spacetimes in the spirit of metric geometry.

In this setting, synthetic curvature conditions have been proposed in analogy with the metric theories of sectional (Alexandrov) and Ricci (RCD) curvature bounds. For \emph{sectional curvature bounds} \cite{kun-sam18,agks21, BORS23}, an analogue of Galloway's splitting theorem has been obtained in \cite{BORS23}. For \emph{Ricci curvature bounds}, introduced in \cite{mccann20,cav-mon22} using optimal transport, general singularity theorems were proven in \cite{cav-mon23} but splitting theorems are not yet available. This is due in part to the lack of a notion of infinitesimal Hilbertianity. In the metric theory this notion arises from the ``analytic'' description of synthetic Ricci bounds via the Bochner inequality  \cite{AGS14} and is known to be a necessary condition for the validity of splitting theorems \cite{gig13,gig14}. Analytic non-smooth Lorentzian geometry is based on \emph{lower gradients} (to appear in the forthcoming work \cite{gigli-et-al}), but its connection to optimal transport is yet to be explored.

In general, despite the keen interest in Lorentzian length spaces and metric techniques in recent years \cite{bur-gar21,kun-ro22,mc-sae22,hau-sau22,be-sae23}, a number of questions remain open, e.g. (i) examples of Lorentzian length spaces, (ii) the relationship between metric and causal geometry, and (iii) splitting theorems under Ricci bounds. Indeed, not many examples of Lorentzian length spaces are known beyond those arising from smooth (or $C^{1,1}$) spacetimes, while the connection of metric products and causal geometry has only been considered in the work of Alexander et. al. \cite{agks21} on Lorentzian warped products -- a special case of \eqref{eq:orth_split} -- for sectional curvature bounds.

\medskip\noindent In this article we construct a generalization of \eqref{eq:orth_split} on product spaces where the hypersurface $S$ is a metric space. Indeed, given a metric space $X$, an interval $I\subset \R$, a continuous function $h:I\times X\to (0,\infty)$, and a family $\mathcal F=\{d_s\}_{s\in I}$ of compatible metrics on $X$, we construct the \emph{generalized Lorentzian product}, denoted $hI\times_\mathcal FX$. 
For it, we establish (1) the push-up principle (Theorem \ref{thm:push-up}) under a log-Lipschitz condition, (2) sufficient conditions for global hyperbolicity (Theorem \ref{thm:glob_hyp}), and (3) partial rigidity of splittings under strong energy conditions (Corollary \ref{cor:partial-rigidity}).

The construction of the generalized Lorentzian product gives new examples of Lorentzian length spaces, cf. Corollary \ref{cor:lor-len-sp} and Theorem \ref{thm:regular-loc-lor-len-sp}. It moreover generalizes the Lorentzian warped product considered in \cite{agks21} and is compatible with \eqref{eq:orth_split} in the setting of manifolds: if $X$ is a smooth manifold and $\mathcal F$ is given by continuous Riemannian metrics, then $hI\times_\mathcal FX$ coincides with the Lorentzian length structure induced by \eqref{eq:orth_split}, cf. Corollary \ref{cor:compatible}.

In contrast to  the warped products of \cite{agks21} (which correspond to $h=1$ and $d_s=f(s)d$ for a continuous function $f:I\to (0,\infty)$ and a fixed metric $d$ on $X$), pathological causal phenomena may occur in the generality of our setting (see e.g. \cite{chru-gra12,terileo}), violating the push-up property required in the definition of Lorentzian length spaces. In Lorentzian manifolds, causal theory is known to break down for metrics below Lipschitz regularity \cite{chru-gra12,grant-kun-sae20}. We identify a log-Lipschitz condition \eqref{eq:lip} on $\mathcal F$ -- inspired by the work of Kopfer--Sturm \cite{kopf-sturm18,kopf-sturm21,kopf19} -- ensuring that the generalized Lorentzian product $hI\times_{\mathcal{F}}X$ is a Lorentzian length space (cf. Corollary \ref{cor:lor-len-sp}). Moreover we show the regularity of $hI\times_\mathcal FX$ for conformal metrics in Theorem \ref{thm:regular-loc-lor-len-sp}, generalizing the corresponding fact for Lipschitz continuous Lorentzian metrics \cite{lange-lytchak-sam21,graf-ling18} in the special case of conformal metrics. Applying synthetic time-like curvature bounds to the generalized Lorentzian product, we obtain the following rigidity statement (Corollary \ref{cor:partial-rigidity}): if $\R\times_\mathcal FX$ equipped with a measure $\bm m=m_s\otimes\ud s$ satisfies the strong energy condition, then the sliced measures $m_s$ are independent of $s$.

\subsection{Generalized Lorentzian product}\label{sec:gen_lor_prod}

Throughout this paper $X$ is a locally compact topological space and $I\subset \R$ an (open) interval. Suppose $h:I\times X\to (0,\infty)$ is continuous, and $\mathcal F=\{d_s\}_{s\in I}$ is a family of metrics on $X$ whose induced length metrics generate the topology of $X$.\footnote{This is the case for example when each metric in $\mathcal F$ is a length metric or a quasiconvex metric.} We make the following local continuity assumption: for each compact $J\subset I$ and $K\subset X$, there exists $r=r_{J,K}>0$ and a modulus of continuity $\omega=\omega_{J, K}:[0,\infty)\to [0,\infty)$ such that 

\begin{equation}\label{eq:cont}
\Big|\log\frac{d_s(x,y)}{d_{s'}(x,y)}\Big|\le \omega(|s-s'|),\quad \textrm{ for }s,s'\in J,\ x,y\in K \textrm{ with } \inf_{s\in J}d_s(x,y)<r.
\end{equation}

In particular all the metrics in $\mathcal F$ are locally bi-Lipschitz equivalent, and the absolute continuity of a curve $\beta:[a,b]\to X$ is independent of the chosen metric in $\mathcal F$. The \emph{generalized metric speed} of an absolutely continuous curve $\beta:[a,b]\to X$  is a map $v_\beta:I\times [a,b]\to [0,\infty]$ with the following properties (cf. Lemma \ref{lem:lip_cont}):
\begin{itemize}
	\item for each $t\in [a,b]$, $s\mapsto v_\beta(s,t)$ is continuous in $I$, and 
	\item for each $s\in I$ we have that $v_\beta(s,t)=|\beta_t'|_{d_s}$ a.e. $t$.\footnote{Here $|\beta_t'|_{d_s}$ denotes the metric speed of $\beta$ with respect to the metric $d_s$.}
\end{itemize}
The \emph{generalized length element} of an absolutely continuous curve $\gamma=(\alpha,\beta):[a,b]\to I\times X$ is
\begin{align}\label{eq:lorentzian-length-element}
	|\tau_\gamma'|^2(t)\defeq h(\gamma_t)^2\alpha'(t)^2-v_\beta(\alpha_t,t)^2,
\end{align}
and we say that $\gamma$ is (a) \emph{causal}, if $|\tau_\gamma'|^2(t)\ge 0$ a.e. $t\in [a,b]$, and (b) \emph{timelike}, if $|\tau_\gamma'|^2(t)> 0$ a.e. $t\in [a,b]$, cf. Definition \ref{def:causal_and_timelike}. A standard construction detailed in Section \ref{sec:causal_structure} gives rise to a causal set $(I\times X,\le,\ll)$ and time separation function $\tau$ on $I\times X$. In the terminology of \cite[Appendix A]{agks21}, the collection $(I\times X,\le,\ll,\tau)$ is a \emph{Lorentzian length structure}, which we denote by $hI\times_\mathcal{F}X$ and call the \emph{generalized Lorentzian product}. In Corollary \ref{cor:length-element-compatible} we show that when $X$ is a smooth manifold and $\mathcal F$ given by continuous Riemannian metrics,  the length elements given by \eqref{eq:orth_split} and \eqref{eq:lorentzian-length-element} coincide. As a direct consequence we obtain the compatibility of $hI\times_\mathcal FX$ and the Lorentzian length structure induced by the Lorentzian product metric.

\begin{cor}\label{cor:compatible}
Let $X$ be a smooth manifold, $h:I\times X\to (0,\infty)$ continuous, and suppose $\{g_s\}_{s\in I}$ is a family of Riemannian metrics on $X$ so that the local coordinate representations $(g_s)_{ij}(p)$ are continuous in $(s,p)$. Let $\mathcal F=\{d_{g_s}\}_{s\in I}$, where $d_{g_s}$ is the length metric associated to $g_s$ for each $s\in I$. Then $hI\times_\mathcal FX$ is the Lorentzian length structure induced by $g\defeq -h^2\ud s^2+g_s$.
\end{cor}

\subsection{Lorentzian length spaces and global hyperbolicity}
Let $h$ and $\mathcal F$ be as in Section \ref{sec:gen_lor_prod}. Under these hypotheses, the generalized Lorentzian product $hI\times_\mathcal FX$ might have causal bubbles (see \cite{terileo}) and thus fail to be a Lorentzian length space, as a result of failing the push-up property. Recall that a causal set $(Y,\le,\ll)$ satisfies the push-up property if \emph{$x\ll z$ whenever $x\le y$ and $y\ll z$ or $x\ll y$ and $y\le z$}. 

To prevent this failure, we impose a log-Lipschitz condition in the spirit of \cite{kopf-sturm18,kopf-sturm21,kopf19}.
\begin{defn}\label{def:loc-log-lip}
We say that $\mathcal F$ satisfies a local log-Lipschitz condition if, for every compact $J\subset I$, $K\subset X$, there exists $C=C_{J,K},r=r_{J,K}>0$ such that 
\begin{align}\label{eq:lip}
	\Big|\log\frac{d_s(x,y)}{d_{s'}(x,y)}\Big|\le C|s-s'|,\quad \textrm{ for }s,s'\in J,\ x,y\in K \textrm{ with } \inf_{s\in J}d_s(x,y)<r.
\end{align}
\end{defn}
In other words, we require that the moduli of continuity $\omega$ in \eqref{eq:cont} are linear, corresponding to a Lipschitz-type control on the time-variable. We remark here that \eqref{eq:lip} in Definition \ref{def:loc-log-lip} is inspired by the work of Kopfer--Sturm on Super Ricci flows on metric spaces \cite{kopf19,kopf-sturm18,kopf-sturm21}, where a similar condition appears. Whereas they focus on the evolution and gradient flows of ``variable metric spaces,'' we instead consider the Lorentzian geometry of the product.

\begin{thm}\label{thm:push-up}
	Suppose $\mathcal F$ satisfies \eqref{eq:lip} and, for every compact $J\subset I$, $K\subset X$ there exists $C'>0$ such that
	\begin{equation}\label{eq:lip-h-time-var}
|h(s,x)-h(s',x)|\le C'|s-s'|,\quad s,s'\in J,\ x\in K.
	\end{equation}
Then $hI\times_{\mathcal F} X$ satisfies the push-up property.
\end{thm}

The push-up property further implies the lower semicontinuity of the associated time separation function $\tau$ on $hI\times_{\mathcal F}X$. These facts together ensure the generalized Lorentzian product is a Lorentzian length space.

\begin{cor}\label{cor:lor-len-sp}
	If $\mathcal F$ and $h$ satisfy \eqref{eq:lip} and \eqref{eq:lip-h-time-var}, respectively, then $hI\times_{\mathcal F} X$ is a Lorentzian length space.
\end{cor}
The Lorentzian length space $hI\times_\mathcal FX$ is moreover strongly local, see Proposition \ref{prop:strongly-local}. However, local conditions such as \eqref{eq:lip} do not guarantee global hyperbolicity. While it is easy to see that $hI\times_\mathcal FX$ is globally hyperbolic if $h$ is bounded and the metrics in $\mathcal F$ are globally bi-Lipschitz to a proper metric (with uniformly bounded bi-Lipschitz constants), examples in \cite{san22} show that properness of the slices does not imply global hyperbolicity, nor does global hyperbolicity imply that the slices are proper. For smooth spacetimes, global hyperbolicity is nevertheless equivalent the existence of a proper metric on the underlying space: in \cite{bur-gar23} this is shown by considering the null-distance associated to a suitable time function. However in our context the natural choice of time function $(t,x)\mapsto t$ does not necessarily yield a proper null-distance, cf. \cite[§ 6.1]{san22}.

In Theorem \ref{thm:glob_hyp} below we give a sufficient condition for the global hyperbolicity in terms of $\mathcal F$ and $h$. The criterion \eqref{eq:inf-length} essentially characterizes the properness of a natural length metric on $I\times X$ associated with $\mathcal F$ and $h$ (see Section \ref{sec:gen-prod-metr}) and is phrased in terms of curves \emph{tending to infinity}. A locally absolutely continuous curve $\beta:[0,\infty)\to X$ is said to tend to infinity if, for every compact $K\subset X$, there exists $T>0$ such that $K\cap \beta([T,\infty))=\varnothing$.

\begin{thm}\label{thm:glob_hyp}
	Suppose $\mathcal F$ and $h$ satisfy \eqref{eq:lip} and \eqref{eq:lip-h-time-var}, and moreover
	\begin{align}\label{eq:inf-length}
		\int_0^\infty\frac{v_\beta(\alpha_t,t)}{h(\alpha_t,\beta_t)}\ud t=\infty
	\end{align}
whenever $\beta:[0,\infty)\to X$ tends to infinity and $\alpha:[0,\infty)\to I$ is locally absolutely continuous. Then $hI\times_\mathcal FX$ is globally hyperbolic.
\end{thm}
Observe that \eqref{eq:inf-length} does not imply the properness of the metrics in $\mathcal F$ but instead of the conformally weighted metrics $(d_s)_{h(s,\cdot)\inv}$; however it is a stronger condition than the propeness of each $(d_s)_{h(s,\cdot)\inv}$, $s\in I$. 

\subsection{Regularity}\label{sec:regularity} The regularity of $hI\times_{\mathcal{F}}X$ seems more difficult to establish. Roughly speaking, a Lorentzian length space is called regular, if \emph{maximizing} curves are either null or time-like (have a definite causal character). These are curves maximizing the Lorentzian length of curves joining given points, and they replace \emph{geodesics} in the current non-smooth framework where no geodesic equation is available. Unlike geodesics, maximizing curves in Lorentzian length spaces need not have a definite causal character; they can have null-segments while having positive Lorentzian length. Regularity is known to hold for spacetimes with a Lorentzian metric of Lipschitz regularity, due to a weak formulation of the geodesic equation \cite{graf-ling18,lange-lytchak-sam21}. However, even in the setting of Lorentzian metrics, \eqref{eq:lip} and \eqref{eq:lip-h-time-var} are weaker than Lipschitz regularity, and we do not know whether they imply regularity of the spacetime in the absence of a suitable Lipschitz condition on the space variable. However we establish regularity of $hI\times_{\mathcal{F}}X$ in the current framework in the special case $\mathcal F=\{d_{\rho(s,\cdot)}\}_{s\in I}$, where $\rho:I\times X\to (0,\infty)$ is a (locally Lipschitz) is a given \emph{conformal weight}. Here, given a length metric $d$ on $X$ and a continuous $g:X\to (0,\infty)$, the conformally weighted metric $d_g$ is a length metric on $X$ such that 
\[
|\beta_t'|_{d_g}=g(\beta_t)|\beta_t'|\quad a.e.\ t
\]
for curves $\beta$ in $X$. The precise definition is given in Section \ref{sec:preli-metric}, compare \eqref{eq:rho-length}.

\begin{thm}\label{thm:regular-loc-lor-len-sp}
Let $d$ be a locally compact length metric on $X$. Suppose that $h,\rho:I\times X\to (0,\infty)$ are locally Lipschitz, and set $\mathcal F=\{d_{\rho(s,\cdot)}\}$. Then $\mathcal F$ satisfies \eqref{eq:lip}, and $hI\times_{\mathcal F}X$ is a regular Lorentzian length space.
\end{thm}

\subsection{Remarks on synthetic time-like Ricci curvature bounds}\label{sec:intro-TCD} The metric techniques used in \cite{gig13,gig14} to prove the splitting theorem of RCD-spaces are not available in the Lorentzian setting as of yet. In the setting of generalized Lorentzian products $hI\times_\mathcal FX$, where the splitting is a priori available, we may instead ask what restrictions do synthetic time-like Ricci curvature bounds place on $h$ and $\mathcal F$, particularly concerning rigidity: if $\R\times_\mathcal FX$ equipped with a measure $\bm m=m_s\otimes\ud s$ has non-negative synthetic time-like Ricci curvature, must $s\mapsto m_s$ and $s\mapsto d_s$ be constant?

Here, $\mathcal M=\{m_s\}_{s\in I}$ is a family of measures on $X$ such that (a) $s\mapsto m_s(E)$ is Borel whenever $E\subset X$ is Borel, (b) $\limsup_{r\to 0}\frac{m_s(B_{d_s}(x,2r))}{m_s(B_{d_s}(x,r))}<\infty$ $m_s$-a.e. $x\in X$ for each $s\in I$ (the infinitesimal doubling property),  and (c) for each compact $J\subset I$ and $K\subset X$, there exists $C=C(J,K)>0$ with
\begin{align}\label{eq:measure-comp}
	m_{s'}|_K\le Cm_s|_K,\quad s,s'\in J.
\end{align}

In Section \ref{sec:curv} we formulate the wTCD$_p$($K,N$)-condition on Lorentzian length structures; the ensuing notion (Definition \ref{def:curv}) agrees with the \cite[Definition 2.20]{braun23} for ($\mathcal K$-)globally hyperbolic regular Lorentzian length spaces. Although the tools to obtain splitting and its rigidity are not available, utilizing the product structure and techniques akin to the needle decomposition in \cite{cav-mon17}, we show that, if $hI\times\mathcal FX$ satisfies the wTCD$_p$($K,N$)-condition, then the space slices $hI\times \{x\}$ are RCD($K,N$)-spaces for a.e. $x\in X$. We refer to Corollary \ref{cor:RCD-lines} for the detailed statement. We remark moreover that, in contrast to \cite{cav-mon17}, no non-branching assumption on $hI\times_\mathcal FX$ is needed due to the product structure available in our setting. As a corollary of Theorem \ref{thm:k,n-convex} we obtain the following partial rigidity of time-like non-negatively curved generalized Lorentzian products.

\begin{cor}\label{cor:partial-rigidity}
Suppose $\mathcal F$ and $\mathcal M$ satisfy  \eqref{eq:cont} and \eqref{eq:measure-comp}, respectively. Let $\bm m=m_s\otimes \ud s$ and suppose $(\R\times_\mathcal FX,\bm m)$ satisfies the wTCD$_p$($0,N$)-condition for $p\in (0,1)$ and $N\ge 1$. Then there is a measure $m$ on $X$ such that $m_s=m$ for a.e. $s\in I$. 

\end{cor}

\subsection*{Organization of the paper} 

We give some background on in Section \ref{sec:preli}: \emph{metric spaces} (Section \ref{sec:preli-metric}), \emph{Lorentzian length spaces} and \emph{structures} (Sections \ref{sec:preli-lorentz} and \ref{sec:preli-lor-st.}), and \emph{time-like Ricci curvature bounds} (Section \ref{sec:curv}). We then define the generalized metric speed and establish its basic properties in Section \ref{sec:gen-metr}.

In Section \ref{sec:lor-len-st} we define the generalized Lorentzian product, based on generalized metric speed. We discuss its causal structure in Section \ref{sec:causal_structure}, the basic properties of causal curves in Section \ref{sec:causal curves}, and prove Corollary \ref{cor:compatible} in Section \ref{sec:manifolds}.

In Section \ref{sec:lorentzian_length_space} we prove Theorem \ref{thm:push-up}, Corollary \ref{cor:lor-len-sp}. The proofs boil down to using Caratheodory theory of ODE's (Appendix \ref{sec:ODE}), which the log-Lipschitz condition \eqref{eq:lip} makes available. Theorem \ref{thm:glob_hyp} is proved in Section \ref{sec:glob_hyp} using a properness criterion given by \eqref{eq:inf-length}. Theorem  \ref{thm:regular-loc-lor-len-sp}, proved in Section \ref{sec:regular}, follows by reducing the problem to Lipschitz regular Lorentzian metrics.

Finally, Section \ref{sec:ricci} contains the discussion of synthetic time-like Ricci curvature bounds on generalized Lorentzian products, and the proof of Corollary \ref{cor:partial-rigidity}.

\subsection*{Acknowledgements} The research in this manuscript was done during my stay at Radboud University and the University of Warwick, as well as the Thematic Program on Nonsmooth Riemannian and Lorentzian Geometry at the Fields Institute. I gratefully acknowledge the financial support of the Radboud Excellence Initiative, the Fields Institute and the ERC (grant no. 948021) of Prof. Bate (Warwick). I would also like to thank Leonardo Garcia-Heveling, Annegret Burtscher, Mathias Braun and Nicola Gigli for many interesting discussions on this manuscript and beyond.

\section{Background}\label{sec:preli}

\subsection{Metric spaces}\label{sec:preli-metric}
Let $(X,d)$ be a metric space. A curve $\beta:[a,b]\to X$ is called absolutely continuous if there exists $g\in L^1(a,b)$ such that 
\begin{align}\label{eq:abs-cont}
d(\beta_s,\beta_t)\le \int_t^sg(r)\ud r,\quad a\le t\le s\le b.
\end{align}
In this case, the limit $\displaystyle |\beta_t'|\defeq \lim_{h\to 0}\frac{d(\beta_{t+h},\beta_t)}{|h|}$ exists a.e. $t\in [a,b]$, and $t\mapsto |\beta_t'|$ is the pointwise a.e. smallest function for which \eqref{eq:abs-cont} holds. We define the length of $\beta$ as $\displaystyle \ell(\beta)=\int_a^b|\beta_t'|\ud t.$ Moreover, there exists a non-decreasing function $s_\beta:[0,\ell(\beta)]\to [a,b]$ and a parametrization $\bar\beta:[0,\ell(\beta)]\to X$, called the \emph{arc-length parametrization} of $\beta$, such that $\beta\circ s_\beta=\bar\beta$ and $|\bar\beta_t'|=1$ for every $t\in [0,\ell(\beta)]$. See \cite{BBI01,haj03}. We say that $d$ is a length metric if 
\begin{align*}
	d(x,y)=\inf\{\ell(\beta)|\ \beta:x\curvearrowright y \},\quad x,y\in X.
\end{align*}

\medskip\noindent To define metrics with a given conformal factor, we first define the line integral of a Borel function $g:X\to [0,\infty]$ along an absolutely continuous curve $\beta:[a,b]\to X$ by
\begin{align*}
	\int_\beta g\,\ud s\defeq \int_a^bg(\beta_t)|\beta_t'|\ud t.
\end{align*}
Note that the line-integral is independent of the parametrization of $\beta$, cf. \cite[Theorem 3.12]{haj03}. Suppose now $d$ is a locally compact length metric on $X$, and let $\rho:X\to (0,\infty)$ be a continuous function. We define the weighted length 
\begin{align*}
	L_\rho(\beta)\defeq \int_\beta\rho\,\ud s
\end{align*}
for any absolutely continuous curve $\beta$ in $X$. Then $L_\rho$ is lower semicontinuous, and the arising length structure $(L_\rho, AC(X))$ on $X$ is compatible with the metric topology of $(X,d)$; the length metric
\begin{align*}
	d_\rho(x,y)=\inf\Big\{\int_a^b\rho(\beta_t)|\beta_t'|\ud t\Big| \ \beta:x\curvearrowright y\},\quad x,y\in X,
\end{align*}
induces the same topology as $d$. Moreover by \cite[Theorem 2.4.3]{BBI01} we have that
\begin{align}\label{eq:rho-length}
|\beta_t'|_{d_\rho}=\rho(\beta_t)|\beta_t'|\quad a.e.\ t\in [a,b]
\end{align}
for every absolutely continuous curve $\beta:[a,b]\to (X,d)$.
\begin{lemma}\label{lem:properness-criterion}
Assume $(X,d)$ is a locally compact length space, and $\rho:X\to (0,\infty)$ continuous. If $\displaystyle \int_0^\infty\rho(\beta_t)|\beta_t'|\ud t=\infty$ whenever $\beta:[0,\infty)\to X$ is locally absolutely continuous and  leaves every compact subset of $X$, then $(X,d_\rho)$ is proper and geodesic.
\end{lemma}
\begin{proof}
By the Hopf--Rinow theorem it suffices to show that $(X,d_\rho)$ is complete. Suppose $(x_j)\subset X$ is a Cauchy sequence for $d_\rho$. The claim follows if we show that $(x_j)$ contains a subsequence converging to a point $x\in X$ with respect to $d_\rho$. Pass to a subsequence (labeled with the same indices) such that $d_\rho(x_j,x_{j+1})<2^{-j-1}$, and let $\beta_j:[j,j+1]\to X$ be an absolutely continuous curve connecting $x_j$ to $x_{j+1}$ constant speed parametrized (with respect to $d$) such that $L_\rho(\beta_j)<d_\rho(x_j,x_{j+1})+2^{-j-1}<2^{-j}$. Consider the concatenation $\beta:[0,\infty)\to X$ defined such that $\beta|_{[0,N]}=\beta_{N-1}\cdots\beta_1\beta_0$. We have that $\beta$ is locally absolutely continuous and
\begin{align*}
	\int_0^\infty\rho(\beta_t)|\beta_t'|_d\ud t=\sum_{j=0}^\infty L_\rho(\beta_j)<2,
\end{align*}
hence-- by our assumption -- it cannot leave every compact set. Consequently there exists a compact set $K\subset X$ such that $(x_j)\subset K$. We conclude that a subsequence of $(x_j)$ converges to $x\in X$ in $d$, and thus in $d_\rho$. 

\end{proof}

\subsection{Lorentzian length spaces}\label{sec:preli-lorentz} The definitions below are taken from \cite{kun-sam18}, which the interested reader should consult for more details. A causal set is a space $Y$ equipped with a transitive and reflexive relation $\le$, together with a transitive relation $\ll$ contained in $\le$. We say that a causal set $(Y,\le,\ll)$ satisfies the push-up property, if \emph{$x\ll z$ whenever $x\le y$ and $y\ll z$ or $x\ll y$ and $y\le z$}.

\begin{defn}\label{def:lor-pre-length}
A Lorentzian pre-length space $Y=(Y,\le,\ll,\tau)$ consists of a causal set $(Y,\le,\ll)$ together with a metric $d$ on $Y$, and a lower semicontinuous time-separation function $\tau:Y\times Y\to [0,\infty]$ such that $\tau(x,y)>0$ if and only if $x\ll y$, satisfying the reverse triangle inequality:
\begin{align*}
	\tau(x,y)+\tau(y,z)\le \tau(x,z)\quad\textrm{if}\quad x\le y\le z.
\end{align*}
\end{defn}
A Lorentzian pre-length space satisfies the push-up property, and the past 
$I^+(x)=\{y: x\ll y \}$ and future $I^-(x)=\{y: y\ll x\}$ of a point $x\in Y$ are open. The corresponding sets with $\ll$ replaced by $\le$ are denoted $J^\pm(x)$. A Lorentzian pre-length space $Y$ is called 
\begin{itemize}
	\item \emph{locally causally closed}, if every point in $Y$ has a neighbourhood $U$ such that $\le$ is closed in $\bar U\times \bar U$;
	\item \emph{strongly causal}, if the metric topology of $Y$ is generated by the chronological diamonds $I(x,y)=I^+(x)\cap I^-(y)$, $x,y\in Y$;
	\item \emph{totally non-imprisoning}, if for any compact $K\subset Y$ there exists $C$ such that $L_\tau(\gamma)\le C$ for any causal curve $\gamma$ in $K$;
	\item \emph{Globally hyperbolic}, if it is totally non-imprisoning and causal diamonds $J(x,y)\defeq J^+(x)\cap J^-(y)$ ($x,y\in Y$) are compact.
\end{itemize}

A curve $\gamma:[a,b]\to Y$ is future-directed causal (resp. time-like) if $\gamma_t\le \gamma_s$ (resp. $\gamma_t\ll\gamma_s$) whenever $a\le t< s\le b$. The \emph{Lorentzian length} or $\tau$-length of a causal curve $\gamma$ is
\begin{align}\label{eq:tau-length}
	L_\tau(\gamma)=\inf\Big\{\sum_{i=1}^N\tau(\gamma_{t_{i-1}},\gamma_{t_i}):\ a=t_0<\cdots< t_N=b \Big\},
\end{align}
and $\gamma$ is called \emph{maximal}, if $L_\tau(\gamma)=\tau(\gamma_a,\gamma_b)$. We call a Lorentzian pre-length space $Y$
\begin{itemize}
	\item[(i)] \emph{causally path connected}, if there exists a causal curve $\gamma:x\curvearrowright y$ whenever $x\le y$, and a time-like curve $\gamma:x\curvearrowright y$ whenever $x\ll y$;
	\item[(ii)] \emph{geodesic}, if there exists a maximal curve $\gamma:x\curvearrowright y$ whenever $x\le y$;
 	\item[(iii)] \emph{localizable}, if every point in $Y$ has a neighbourhood $U$ such that
 	\[ L_\tau^U(p,q)\defeq \sup\{L_\tau(\gamma):\ \gamma:p\curvearrowright q \textrm{ causal, and } \gamma\subset U \},\quad p,q\in U
 	\]
 	is bounded from above and $L_\tau^U\le \tau|_{U\times U}$, and moreover  $Y_U\defeq (U,\le_{U\times U},\ll_{U\times U},L_\tau^U)$ is a geodesic Lorentzian pre-length space;
 	\item[(iv)] \emph{locally regularizable}, if the neighbourhood $U$ above can be chosen so that maximizing curves in $Y_U$ with positive $\tau$-length are time-like.
\end{itemize}
A neighbourhood $U$ as in (iii) and/or (iv) is called a localizing neighbourhood. Locally regularizability ensures that maximizing curves have a definite causal character, i.e. they are either time-like or null (zero $\tau$-length), cf. \cite[Theorem 3.18]{kun-sam18}.

\begin{defn}(Definition 2.8 and 3.22 in \cite{kun-sam18})\label{def:lor-len-sp}
A localizable, locally causally closed, and causally path connected Lorentzian pre-length space $(Y,\le,\ll,\tau)$ is called a 
\begin{itemize}
\item[(1)] \emph{Lorentzian length space}, if $\tau(x,y)=\sup\{L_\tau(\gamma):\ \gamma:x\curvearrowright y\textrm{ causal} \}$ whenever $x\le y$;
\item[(2)] \emph{regular Lorentzian length space}, if it is a locally regularizable Lorentzian length space;
\item[(3)] \emph{Lorentzian geodesic space}, if it is a geodesic Lorentzian length space.
\end{itemize} 
\end{defn}

\subsection{Lorentzian length structures}\label{sec:preli-lor-st.} We sketch the idea of Lorentzian lengths structures here and refer to \cite[Appendix A]{agks21} for the precise definitions. Suppose $\mathcal I^+ \subset \mathcal C^+\subset AC(Y,d)$ are families of curves in a metric space $(Y,d)$ closed under concatenation, restriction, and reparametrization. If $\stackrel{\leftarrow}{\gamma}$ denotes the curve $\gamma$ with reversed orientation and $\mathcal I^-=\{\stackrel{\leftarrow}{\gamma}:\ \gamma\in \mathcal I^+\}$, $\mathcal C^-=\{\stackrel{\leftarrow}{\gamma}:\ \gamma\in \mathcal C^+\}$, then $(\mathcal I^\pm,\mathcal C^\pm)$ is  admissible in the sense of \cite[Definition A.1]{agks21}. For the following definition set $\mathcal I\defeq \mathcal I^-\cup\mathcal I^+$, $\mathcal C\defeq \mathcal C^-\cup\mathcal C^+$.

\begin{defn}\label{def:lor-len-str}
A Lorentzian length structure on $Y$ consists of such a tuple $(\mathcal I^\pm,\mathcal C^\pm)$ together with a Lorentzian length functional $L:\mathcal C\to [0,\infty]$ that is additive, invariant under reparametrization, strictly positive on $\mathcal I$, and compatible with the topology: if $\gamma:[a,b]\to Y$ is in $\mathcal C$, then $t\mapsto L(\gamma|_{[a,t]})$ is continuous.
\end{defn}
We call curves in $\mathcal I^\pm$ (resp. $\mathcal C^\pm$) past/future time-like (resp. causal). A Lorentzian length structure $(Y,\mathcal I^\pm,\mathcal C^\pm,L)$ induces a causal set $(Y,\le,\ll)$, where $x\le y$ if there exists $\gamma:x\curvearrowright y$ in $\mathcal C^+$, and $x\ll y$ if there exists $\gamma:x \curvearrowright y$ in $\mathcal I^+$, cf. \cite[Lemma A.4]{agks21}. Moreover,
\begin{align*}
	\tau_L(x,y)=\sup\{L(\gamma):\ \gamma:x\curvearrowright y\textrm{ is in }\mathcal C^+\}
\end{align*}
(with the convention $\sup \varnothing =0$) defines a time-separation function which is positive on $\ll$ and satisfies the reverse triangle inequality, cf. \cite[Lemma A.6 and A.7]{agks21}. The induced tuple $(Y,\le,\ll,\tau_L)$ is causally path connected, and satisfies the length condition in Definition \ref{def:lor-len-sp}(1). However $\tau$ is not lower semicontinuous and the push-up property may not hold in general.

\begin{remark}\label{rmk:lor-len-sp-vs-str}
If $(Y,\le,\ll,\tau)$ is a Lorentzian length space and $\mathcal I^+$, $\mathcal C^+$ are the families of future directed time-like and causal curves, respectively, and if $L$ is given by \eqref{eq:tau-length}, then $(\mathcal I^\pm,\mathcal C^\pm,L)$ is a Lorentzian length structure whose induced causal set coincides with $(Y,\le,\ll)$, and $\tau=\tau_L$.
\end{remark}

\subsection{Synthetic time-like Ricci curvature bounds}\label{sec:curv} We define the (weak) TCD$_p$($K,N$)-condition for Lorentzian length structures $Y$ for $p\in (0,1)$, $K\in \R$ and $N\in [1,\infty)$ following the ideas in \cite{cav-mon22,braun23}.

Given two measures $\mu,\nu\in \mathcal P(Y)$, a coupling of $(\mu,\nu)$ is a measure $\bm\pi\in \mathcal P(Y\times Y)$ such that $(p_{1\ast}\bm\pi,p_{2\ast}\bm\pi)=(\mu,\nu)$. A coupling is causal if it is concentrated on $\le$, and their collection is denoted $\Pi_\le(\mu,\nu)$. Causal couplings are closely related with \emph{causal plans}, that is measures $\bm\eta\in \mathcal P(C([0,1],Y))$ concentrated on $\mathcal C^+$: if $\bm\eta$ is a causal plan, then $(e_0,e_1)_\ast\bm\eta$ is a causal coupling of $(e_{0\ast}\bm\eta,e_{1\ast}\bm\eta)$. Here $e_t:C([0,1],Y)\to Y$ is the evaluation map $\gamma\mapsto \gamma_t$ for $t\in [0,1]$. 
\begin{defn}\label{def:KN-convexity}
	A functional $S:\mathcal P(Y)\to [-\infty,\infty]$ is $(K,N)$-convex along a causal plan $\bm\eta$, if the function $t\mapsto u(t)\defeq \exp(-S(e_{t\ast}\bm\eta)/N)$ satisfies
	\begin{align*}
		u(t)\ge \sigma_{K/N}^{1-t}(T_{\bm\eta})u(0)+\sigma_{K/N}^{t}(T_{\bm\eta})u(1),\quad \textrm{where}\quad T_{\bm\eta}=\bigg(\int L(\gamma)^2\ud\bm\eta(\gamma)\bigg)^{1/2}.
	\end{align*}
Equivalently, $u$ is semiconvex and satisfies $\displaystyle u''(t)\le -\frac KNT_{\bm\eta}\,u(t)$ a.e. $t\in [0,1]$. 
\end{defn}
In the definition above we use the auxiliary functions
\begin{align*}
	\sigma_\kappa^t(\theta)\defeq\left\{
	\begin{array}{ll}
		\frac{s_\kappa(t\theta)}{s_\kappa(\theta)} & \kappa\theta^2<0\textrm{ or }\kappa\theta^2\in (0,\pi^2)\\
		t & \kappa\theta^2=0\\
		\infty & \kappa\theta^2\ge \pi^2
	\end{array}
	\right.
	,\quad
	s_\kappa(\theta)\defeq\left\{
	\begin{array}{ll}
		\frac{\sin(\sqrt\kappa \theta)}{\sqrt\kappa} & \kappa>0 \\
		\theta & \kappa= 0 \\
		\frac{\sinh(\sqrt{-\kappa}\theta)}{\sqrt{-\kappa}} & \kappa<0.
	\end{array}
	\right.
\end{align*}

Next we define $p$-optimal time-like plans. Given $p\in (0,\infty)$, define the Lorentz--Wasserstein $p$-distance
\begin{align*}
	\ell_p(\mu,\nu)=\sup\Big\{\Big(\int\tau_L(x,y)\ud\bm\pi(x,y)\Big)^{1/p}:\ \bm\pi\in \Pi_\le(\mu,\nu) \Big\},
\end{align*}
with the convention $\sup\varnothing=0$. Moreover, let ${\rm OpT}(Y)$ denote the family of curves $\gamma\in \mathcal I^+$ such that
\begin{align*}
	\tau(\gamma_s,\gamma_t)=(s-t)L(\gamma),\quad a\le t\le s\le b. 
\end{align*}
Note that, since ${\rm OpT}(Y)\subset \mathcal I\subset AC(Y,d)$, its elements are automatically continuous.
\begin{defn}\label{def:p-optimal-plan}
	A $p$-optimal time-like plan is a causal plan concentrated on ${\rm OpT}(Y)$ such that
\begin{align*}
	\ell_p(e_{0\ast}\bm\eta,e_{1\ast}\bm\eta)^p=\int L(\gamma)^p\ud\bm\eta(\gamma).
\end{align*}
If $e_{0\ast}\bm\eta=\mu$ and $e_{1\ast}\bm\eta=\nu$, $\bm\eta$ is called a $p$-optimal plan for $(\mu,\nu)$, and their collection is denoted ${\rm OpT}_p(\mu,\nu)$.
\end{defn}

\begin{remark}
	Let $p\in (0,1)$. If $Y$ is a ($\mathcal K$-)globally hyperbolic regular Lorentzian length space, then every strongly time-like $p$-dualizable pair $(\mu,\nu)$ of probability measures on $Y$ admits a $p$-optimal time-like plan for $(\mu,\nu)$, see \cite[Lemma 2.13]{braun23}.
\end{remark}

\begin{defn}\label{def:curv}
	Let $(Y,\mathcal I^\pm,\mathcal C^\pm,L)$ be a Lorentzian length structure and $\bm m$ be a Radon measure on $Y$. Let $p\in (0,1)$, $K\in \R$ and $N\in [1,\infty)$. We say that $(Y,\bm m)$ satisfies the wTCD$_p$($K,N$)-condition if, for any $\mu,\nu\in \mathcal P_{c}^{ac}(Y,\bm m)$ with ${\rm OpT}_p(\mu,\nu)\ne\varnothing$, there exists $\bm\eta\in {\rm OpT}_p(\mu,\nu)$ such that ${\rm Ent}_{\bm m}$ is $(K,N)$-convex along $\bm\eta$. 
\end{defn}
\noindent Here, the entropy functional ${\rm Ent}_{\bm m}:\mathcal P(Y)\to [-\infty,\infty]$ is given by
\begin{eqnarray*}
{\rm Ent}_{\bm m}(\mu)=\left\{
\begin{array}{ll}
\displaystyle \int\rho\log\rho\ud\bm m, & \mu=\rho\bm m\textrm{ with }(\rho\log\rho)_+\in L^1(\bm m) \\
\displaystyle +\infty, &\textrm{otherwise}
\end{array}
\right.
\end{eqnarray*}

\section{A generalized product metric}\label{sec:gen-metr}

\subsection{Generalized metric speed}\label{sec:gen_metr_sp}

Let $X$ be a topological space, $I\subset \R$ an interval, and $\mathcal F=\{d_s\}_{s\in I}$ a family of metrics on $X$ satisfying \eqref{eq:cont}.

\begin{defn}
Let $\beta:[a,b]\to X$ be an absolutely continuous curve. We define $v_\beta:I\times [a,b]\to \R$ by
\begin{align*}
	v_\beta(s,t)=\limsup_{h\to 0}\frac 1h\int_{t}^{t+h}|\beta'_\tau|_{d_s}\ud \tau.
\end{align*}
\end{defn}

\begin{lemma}\label{lem:lip_cont}
	Suppose $J\subset I$ and $K\subset X$ are compact. Let $\beta:[a,b]\to K$ be an absolutely continuous curve. There exists $C=C(J,K)$ such that 
	\begin{itemize}
		\item[(1)] for every $s\in I$, $v_\beta(s,t)=|\beta_t'|_{d_s}$ a.e. $t\in [a,b]$;
		\item[(2)] for every compact subinterval $J\subset I$ we have that
		\[
		\frac{v_\beta(s,t)}{v_\beta(s',t)}\le 1+C\omega(|s'-s|),\quad t\in [a,b],\ s,s'\in J,
		\]
		where $\omega=\omega_{J,K}$ is the modulus of continuity in \eqref{eq:cont}.

	\end{itemize}
\end{lemma}

In the proof we use the observation $e^A\le 1+C(D)A$ for $A\in [0,D]$, which implies that \eqref{eq:cont} can be equivalently expressed as
\begin{align}\label{eq:contv2}
	\frac{d_s(x,y)}{d_{s'}(x,y)}\le 1+C\omega_{J,K}(|s-s'|)\quad \textrm{ for }s,s'\in J,\ x,y\in K \textrm{ with } \inf_{s\in J}d_s(x,y)<r,
\end{align}
for some constants $C=C_{J,K},r=r_{J,K}>0$ for each compact $J\subset I$ and $K\subset X$. If $\mathcal F$ satisfies \eqref{eq:lip}, the right hand side of \eqref{eq:contv2} can be replaced by $1+C|s-s'|$. In light of this, Lemma  \ref{lem:lip_cont} has the following immediate corollary.
\begin{cor}\label{cor:lip_cont}
	Suppose $\mathcal F$ satisfies \eqref{eq:lip}. For every compact $K\subset X$ and $J\subset I$, there exists $C=C_{K,J}$ such that
	\begin{align*}
		\frac{v_\beta(s',t)}{v_\beta(s,t)}\le 1+C|s'-s|,\quad t\in [a,b],\ s,s'\in J,
	\end{align*}
	for any absolutely continuous curve $\beta:[a,b]\to K$.
\end{cor}

\begin{proof}
The first claim follows from Lebesgue's differentiation theorem since, for fixed $s\in I$, the map $t\mapsto |\beta_t'|_{d_s}$ is $L^1$-integrable by the absolute continuity of $\beta$ with respect to $d_s$.

By \eqref{eq:contv2} we have that 
	\begin{align*}
		\frac{\ell_{d_s}(\beta|_{[t,t+h]})}{h}\le \frac{\ell_{d_{s'}}(\beta|_{[t,t+h]})}{h}(1+C\omega(|s-s'|))
	\end{align*}
and, upon taking limsup as $h\to 0$,
\[
v_\beta(s,t)\le v_\beta(s',t)(1+C\omega(|s-s'|))
\]
for $s,s'\in J$. This proves (2).
\end{proof}

The following change of variables formula for $v_\beta$ will be very useful. 

\begin{lemma}\label{lem:reparam}
	Let $\beta:[a,b]\to X$ be an AC-curve, and $\varphi:[c,d]\to [a,b]$ an absolutely continuous bijection. Then
	\[
	v_{\beta\circ\varphi}(s,t)=|\varphi'(t)|v_\beta(s,\varphi(t))
	\]
	for all $s\in I$ and $t\in [a,b]$ for which $\varphi'(t)$ exists.
\end{lemma}
In particular if $\alpha:[a,b]\to I$ is an absolutely continuous curve we have that
\[
v_{\beta\circ\varphi}(\alpha(t),t)=|\varphi'(t)|v_\beta(\alpha(t),\varphi(t))\quad a.e.\ t.
\]
\begin{proof}
	We may assume that $\varphi$ is increasing. Let $s\in I$ and observe that
\begin{align*}
	v_{\beta\circ\varphi}(s,t)=\limsup_{h\to 0}\frac{\ell_s(\beta|_{[\varphi(t),\varphi(t)+\delta(h)]})}{h},
\end{align*}
where $\delta(h):=\varphi(t+h)-\varphi(t)$. If $\varphi'(t)$ exists we thus have that 
\begin{align*}
	v_{\beta\circ\varphi}(s,t)=\limsup_{h\to 0}\frac{\delta(h)}{h}\frac{\ell_s(\beta|_{[\varphi(t),\varphi(t)+\delta(h)]})}{\delta(h)}=\varphi'(t)\limsup_{\delta\to 0}\frac{\ell_s(\beta|_{[\varphi(t),\varphi(t)+\delta]})}{\delta} =\varphi'(t)v_\beta(s,\varphi(t)).
\end{align*}
\end{proof}

\begin{lemma}\label{lem:weak-conv-gen-speed}
	Suppose a sequence $\gamma_j=(\alpha_j,\beta_j):[a,b]\to I\times X$ of AC-curves converges uniformly to an AC-curve $\gamma=(\alpha,\beta)$. Then
	\begin{align*}
		\int_c^dv_\beta(\alpha(t),t)\ud t\le \liminf_{j\to\infty}\int_c^dv_{\beta_j}(\alpha_j(t),t)\ud t
	\end{align*}
whenever $a\le c\le d\le b$.
\end{lemma}
\begin{proof} 
	By considering the restriction to $[c,d]$ we may assume that $c=a$ and $d=b$. Fix $\varepsilon>0$. By uniform convergence $(\alpha_j)$ is equicontinuous and thus there exists $\delta>0$ such that \linebreak $\sup_{j}|\alpha_j(t)-\alpha_j(s)|\le \varepsilon$ whenever $|t-s|<\delta$. Let $a=a_0<\cdots <a_M=b$ be a partition of $[a,b]$ such that $a_i-a_{i-1}<\delta$, $i=1,\ldots, M$. For each $i$ we have by Lemma \ref{lem:lip_cont}(2) that 
	\begin{align*}
	\int_{a_{i-1}}^{a_i} v_\beta(\alpha(t),t)\ud t &\le (1+\omega(\varepsilon))	\int_{a_{i-1}}^{a_i}v_{\beta}(\alpha(a_i),t)\ud t\le (1+\omega(\varepsilon))\liminf_{j\to\infty} \int_{a_{i-1}}^{a_i}v_{\beta_j}(\alpha(a_i),t)\ud t \\
	&\le (1+\omega(\varepsilon))^2\liminf_{j\to\infty} \int_{a_{i-1}}^{a_i}v_{\beta_j}(\alpha_j(t),t)\ud t
	\end{align*}
Summing over $i$ and taking the limit $\varepsilon\to 0$ we have the estimate
\[
\int_a^bv_\beta(\alpha(t),t)\ud t\le \liminf_{j\to\infty}\int_a^bv_{\beta_j}(\alpha_j(t),t)\ud t, 
\]
as claimed.
\end{proof}

\subsection{Generalized product metric}\label{sec:gen-prod-metr}

We define the generalized product metric $d_{\mathcal F,h}$ on $I\times X$ as the length metric arising from the length functional
\begin{align*}
L_{\mathcal F,h}(\gamma)=\int_a^b\sqrt{h(\gamma_t)\alpha'(t)^2+v_\beta(\alpha_t,t)^2}\ud t,\quad \gamma=(\alpha,\beta):[a,b]\to I\times X,
\end{align*}
see Section \ref{sec:preli-metric}. We note that $d_{\mathcal F,h}$ generates the product topology on $I\times X$, and $(I\times X,d_{\mathcal F,h})$ is therefore a locally compact length space. 
Note that, if $J\subset I$ is a subinterval (we will mostly consider compact subintervals) and $\mathcal F_J=\{d_s\}_{s\in J}$, then $(J\times X,d_{\mathcal F_J,h})$ is homeomorphic with $(J\times X,d_{\mathcal F,h})$ but the metrics need not coincide.

Another convenient choice of background metric on $I\times X$ in what follows will be the conformally weighted metric $(d_{\mathcal F,h})_{h\inv}$, see Section \ref{sec:preli-metric}. We have that  $(d_{\mathcal F,h})_{h\inv}$ generates the product topology, and a curve $\gamma=(\alpha,\beta):[a,b]\to I\times X$ is in $AC(X,d_{\mathcal F,h})$ if and only if it is in $AC(X,(d_{\mathcal F,h})_{h\inv})$, and in this case
\begin{align*}
	|\gamma_t'|_{(d_{\mathcal F,h})_{h\inv}}=\frac{|\gamma_t'|_{d_{\mathcal F,h}}}{h(\gamma_t)}=\sqrt{\alpha'(t)^2+\frac{v_\beta(\alpha_t,s)^2}{h(\gamma_t)^2} }\quad a.e.\ t\in [a,b].
\end{align*}
If not explicitly stated otherwise, we tacitly equip $I\times X$ with the metric $(d_{\mathcal F,h})_{h\inv}$.

\section{The Lorentzian length structure on $hI\times_{\mathcal{F}}X$}\label{sec:lor-len-st}

\subsection{Causal structure}\label{sec:causal_structure}

Let $X$ and $\mathcal F$ be as in Section \ref{sec:gen_metr_sp} -- we assume in particular that $\mathcal F$ satisfies \eqref{eq:cont} -- and let $h:I\times X\to (0,\infty)$ be continuous. Throughout this subsection we denote $Y=I\times X$.

\subsubsection*{Definitions}
An absolutely continuous curve $\gamma=(\alpha,\beta):[a,b]\to Y$ is said to be \emph{future directed} if $\alpha'>0$ a.e. and \emph{past directed} if $\alpha'<0$ a.e.
\begin{defn}\label{def:causal_and_timelike}
	Let $\gamma=(\alpha,\beta):[a,b]\to Y$ be a future or past directed curve. We say that $\gamma$ is
	\begin{itemize}
		\item[(a)]  {\bf causal} if $v_\beta(\alpha_t,t)^2\le h(\gamma_t)^2\alpha'(t)^2$ a.e. $t\in [a,b]$;
		\item[(b)] {\bf timelike} if $v_\beta(\alpha_t,t)^2< h(\gamma_t)^2\alpha'(t)^2$ a.e. $t\in [a,b]$;
	\end{itemize}
\end{defn}
We also say that $\gamma$ is a \emph{null curve} if $v_\beta(\alpha_t,t)^2= h(\gamma_t)\alpha'(t)^2$ a.e. $t\in [a,b]$. Note that by Lemma \ref{lem:reparam} the notions above are invariant under absolutely continuous orientation preserving reparametrizations. Given a causal curve $\gamma=(\alpha,\beta):[a,b]\to Y$ recall its {\bf Lorentzian length} given by
\begin{align}\label{eq:lorentzian_length}
	\tau(\gamma)=\int_a^b\sqrt{h(\gamma(t))^2\alpha'(t)^2-v_\beta(\alpha(t),t)^2}\ud t.
\end{align}
Lorentzian length is independent of parametrization (cf. Lemma \ref{lem:reparam}) and additive under concatenation of (causal) curves.
\begin{remark}\label{rmk:causal-curve-length-bd}
	If $\gamma=(\alpha,\beta):[a,b]\to hI\times_\mathcal FX$ is a (future or past directed) causal curve, then 
	\begin{align*}
		\ell_{(d_{\mathcal F,h})_{h\inv}}(\gamma)\le \sqrt 2|\alpha_b-\alpha_a|.
	\end{align*}
	Indeed, 
	\begin{align*}
		\ell_{(d_{\mathcal F,h})_{h\inv}}(\gamma)=\int_a^b\sqrt{\alpha'(t)^2+\frac{v_\beta(\alpha_t,t)^2}{h(\gamma_t)^2}}\ud t\le \sqrt 2\int_a^b|\alpha'(t)|\ud t=\sqrt 2|\alpha_b-\alpha_a|.
	\end{align*}
\end{remark}

We now define the causal structure and time separation function on $Y$.

\begin{defn}\label{def:causal_structure}
	Let $p,q\in Y$. We say that $p$ and $q$ are
	\begin{itemize}
		\item[(a)] {\bf causally} related, denoted $p\le q$, if there exists a future directed causal curve $\gamma:p\curvearrowright q$;
		\item[(b)] {\bf timelike} related, denoted $p\ll q$, if there exists a future directed timelike curve $\gamma:p\curvearrowright q$;
	\end{itemize}
\end{defn}
A standard argument yields that $\le$ is a preorder and $\ll$ is a transitive relation contained in $\le$, so that $(Y,\ll,\le)$ is a \emph{causal set}, cf. \cite[Definition 2.1]{kun-sam18}. The causal and timelike \emph{future} and \emph{past} of a point $p\in Y$ are denoted
\begin{align*}
	I^+(p)=\{q\in Y:\ p\ll q\},\quad J^+(p)=\{q\in Y:\ p\le q\}\\
	I^-(p)=\{q\in Y: q\ll p\},\quad J^-(p)=\{q\in Y: q\le p\},
\end{align*}
and the chronological and causal diamonds $I^\pm (p,q)\defeq I^+(p)\cap I^-(q)$, $J^\pm (p,q)\defeq J^+(p)\cap J^-(q)$. The time separation function is defined in a standard way.
\begin{defn}\label{def:time_separation_function}
	The {\bf time separation function} $\tau:Y\times Y\to \R$ is defined as 
	\begin{align*}
		\tau(p,q)=\sup_{\gamma:p\curvearrowright q}\tau(\gamma),
	\end{align*}
	where the supremum is taken over all causal curves joining $p$ and $q$ (with the convention $\sup \varnothing=0$).
\end{defn}
\begin{remark}\label{rmk:pre_pre_lorentz_length_sp}
	The time separation function satisfies the following two standard properties:
	\begin{itemize}
		\item[(1)] If $p\ll q$ then $\tau(p,q)>0$;
		\item[(2)] If $p\le z\le q$ then $\tau(p,z)+\tau(z,q)\le \tau(p,q)$. 
	\end{itemize}
	However we do not know at this stage whether $\tau$ is lower semicontinuous. We will establish this in the next section assuming \eqref{eq:lip}.
\end{remark}

\subsection{Basic properties of causal curves}\label{sec:causal curves}

We first establish the stability of causal curves and the upper semicontinuity of Lorentzian length.

\begin{prop}\label{prop:causality_stable}
	Suppose $\gamma_j=(\alpha_j,\beta_j):[a,b]\to Y$ are causal curves converging uniformly to an AC-curve $\gamma=(\alpha,\beta):[a,b]\to Y$. If  $\alpha'>0$ (or $<0$) then $\gamma$ is a causal curve.
\end{prop}
\begin{proof}
Without loss of generality we assume that $\alpha'>0$ a.e. We will prove the following equivalent form of the causality condition:
\begin{align}\label{eq:equiv-causal}
\int_{a'}^{b'}v_\beta(\alpha(\tau),\tau)\ud \tau\le \int_{a'}^{b'}h(\gamma(\tau))\alpha'(\tau)\ud \tau,\quad a\le a'<b'\le b.
\end{align}
By Lemma \ref{lem:weak-conv-gen-speed} we have that
\[
\int_{a'}^{b'}v_\beta(\alpha(\tau),\tau)\ud \tau\le\liminf_{j\to\infty}\int_{a'}^{b'}v_{\beta_j}(\alpha_j(\tau),\tau)\ud \tau.
\]
Fix $\varepsilon>0$. The causality of $\gamma_j$ yields
\begin{align*}
	\int_{a'}^{b'}v_{\beta_j}(\alpha_j(\tau),\tau)\ud \tau \le \int_{a'}^{b'}h(\gamma_j(\tau))\alpha_j'(\tau)\ud \tau\le \int_{a'}^{b'}(h(\gamma(\tau))+\varepsilon)\alpha_j'(\tau)\ud \tau
\end{align*}
for large enough $j$. Since $\alpha_j'(t)\ud t\to \alpha'(t)\ud t$ weakly as measures we obtain that 
\[
\int_{a'}^{b'}(h(\gamma(\tau))+\varepsilon)\alpha_j'(\tau)\ud \tau\to \int_{a'}^{b'}(h(\gamma(\tau))+\varepsilon)\alpha'(\tau)\ud \tau.
\]
Letting $\varepsilon\to 0$ we obtain \eqref{eq:equiv-causal}.
\end{proof}

\begin{prop}\label{prop:lorentz_usc}
Lorentzian length is upper semicontinuous, i.e. if $\gamma_j$ is a sequence of causal curves converging uniformly to a causal curve $\gamma$ then
\[
\limsup_{j\to\infty}\tau(\gamma_j) \le\tau(\gamma).
\]
\end{prop}

\begin{proof}
	By passing to a subsequence we may assume that $\limsup_{j\to\infty}\tau(\gamma_j)$ is a limit. By the invarience of reparametrization we may assume that $\gamma_j(t)=(c_jt,\beta_j(t))$, $t\in [0,1]$, and that $c_j\to c$, $\beta_j\to \beta$ (uniformly) and
	\[\tau(\gamma)=\int_0^1\sqrt{c^2h(ct,\beta(t))^2-v_{\beta}(ct,t)^2}\ud t.\]
	Since $v_{\beta_j}(c_jt,t)\le c_jh(c_jt,\beta_j(t))$ a.e. we may pass to a (further) subsequence such that $g_j\defeq v_{\beta_j}(c_j\cdot,\cdot)$ and $f_j\defeq \sqrt{c_j^2h(c_j\cdot,\beta_j(\cdot))^2-g_j^2}$ converge weakly in $L^2([0,1])$ to $g$  and $f$, respectively. The lower semicontinuity of the $L^2$-norm with respect to weak convergence implies that, for every $0\le a<b\le 1$ we have 
	\[
	\int_a^bg^2\ud t\le \liminf_{j\to \infty}\int_a^bg_j^2\ud t,\quad 	\int_a^bf^2\ud t\le \liminf_{j\to \infty}\int_a^bf_j^2\ud t
	\]
	and thus
	\[
	\int_a^b(g^2+f^2)\ud t\le \liminf_{j\to\infty}\int_a^b(g_j^2+f_j^2)\ud t=c^2\int_a^bh(ct,\beta(t))^2\ud t,
	\]
	By Lemma \ref{lem:weak-conv-gen-speed} we have that
	\[
	\int_a^bv_\beta(ct,t)\ud t\le\liminf_{j\to\infty}\int_a^bv_{\beta_j}(c_jt,t)\ud t=\liminf_{j\to\infty}\int_a^bg_j\ud t=\int_a^bg\ud t.
	\]
	Since $a,b$ are arbitrary it follows that
	\[
	v_\beta(ct,t)^2+h(t)^2\le g(t)^2+f(t)^2\le c^2h(ct,\beta(t))^2\quad a.e.\ t\in [0,1].
	\]
	From this we obtain that
	\begin{align*}
		\tau(\gamma)=\int_0^1\sqrt{c^2h(ct,\beta(t))^2-v_\beta(ct,t)^2}\ud t\ge \int_0^1h\ud t=\lim_{j\to \infty}\int_0^1h_j\ud t=\lim_{j\to\infty}\tau(\gamma_j).
	\end{align*}
\end{proof}
We single out the following corollary of Propositions \ref{prop:causality_stable} and \ref{prop:lorentz_usc}.

\begin{cor}\label{cor:loc-comp-causal-diamond}
	\begin{itemize}
		\item[(a)]  Assume $p,q\in I\times X$ and $J(p,q)$ is contained in a compact set. Then $J(p,q)$  is compact and, for every $p',q'\in J(p,q)$, there exists a maximizing causal curve $\gamma:p'\curvearrowright q'$.
		\item[(b)] Let $p_0\in I\times X$ and $h>0$. If $\overline B_{(d_{\mathcal F,h})_{h\inv}}(p_0,4h)$ is compact, then $J(p,q)$ is compact for every $p,q\in \overline B_{(d_{\mathcal F,h})_{h\inv}}(p_0,h)$.
	\end{itemize}
\end{cor}
\begin{proof}
	It follows from Remark \ref{rmk:causal-curve-length-bd} that
	\[
	(d_{\mathcal F,h})_{h\inv}(p,q')+(d_{\mathcal F,h})_{h\inv}(q',q)\le \sqrt 2(b-a)
	\]
	for any $q'\in J(p,q)$, whenever $p=(a,x)$ and $q=(b,y)$ are causally related. Assume $J(p,q)$ is contained in a compact set. To show that $J(p,q)$ is compact it suffices to show it is closed. If $(q_j')\subset J(p,q)$ converges to $q'$, let $\gamma_j^p:p\curvearrowright q_j'$ and $\gamma_j^q:q_j'\curvearrowright q$ be causal curves parametrized with affine first component. 
	We have that the curves $\gamma_j^p$ and $\gamma_j^q$ lie inside a compact set and have uniformly bounded length. Thus the sequences converge up to a subsequence to causal curves $\gamma^p:p\curvearrowright q'$ and $\gamma_q:q'\curvearrowright q$, respectively. This shows that $q'\in J(p,q)$, and thus $J(p,q)$ is closed. Moreover, if $p',q'\in J(p,q)$ and $\gamma_j:p'\curvearrowright q'$ is a sequence of causal curves with $\tau(\gamma_j)\to \tau(p',q')$ for $p',q'\in J(p,q)$, then up to reparametrization a subsequence of $(\gamma_j)$ converges to a causal curve $\gamma:p'\curvearrowright q'$ satisfying $\tau(p',q')=\lim_{j\to\infty}\tau(\gamma_j)\le \tau(\gamma)$. Thus $\gamma$ is a maximizing curve. We have proved (a).
	
	To prove (b) note that, if $p,q\in \overline B_{(d_{\mathcal F,h})_{h\inv}}(p_0,h)$, then $|b-a|\le 2h$, and thus
	\[
	(d_{\mathcal F,h})_{h\inv}(p_0,q')\le (d_{\mathcal F,h})_{h\inv}(p_0,p)+(d_{\mathcal F,h})_{h\inv}(p,q')\le (1+2\sqrt 2)h<4h.
	\]
	This shows that $J(p,q)\subset B_{(d_{\mathcal F,h})_{h\inv}}(p_0,4h)$ for all $p,q\in \overline B_{(d_{\mathcal F,h})_{h\inv}}(p_0,h)$. By (a) $J(p,q)$ is compact.
\end{proof}

\subsubsection*{Variational length of causal curves}
Recall the Lorentzian length \eqref{eq:tau-length} associated to the time separation function $\tau$.

Our next proposition states that it agrees with \eqref{eq:lorentzian_length}.
\begin{prop}\label{prop:lor=var}
	Let $\gamma:[a,b]\to Y$ be a causal curve. Then $L_\tau(\gamma)=\tau(\gamma).$
\end{prop}
\begin{proof}
The inequality $\tau(\gamma)\le L_\tau(\gamma)$ is a straightforward consequence of the definitions. We will prove the opposite inequality.

By reparametrization we may assume that $\gamma=(\id,\beta):[a,b]\to Y$. For each $t\in [a,b]$ and $h>0$ such that $t+h\in [a,b]$, let $\gamma_{t,h}=(\id,\beta_{t,h})$ be a maximizing geodesic $\gamma_t\curvearrowright \gamma_{t+h}$. Using H\"older's inequality
\begin{align*}
	&\tau(\gamma_t,\gamma_{t+h})^2\le h\int_t^{t+h}[h(s,\beta_{t,h}(s))-v_{\beta_{t,h}}(s,s)^2]\ud s,\quad \textrm{ or equivalently}\\
	&\frac{\tau(\gamma_t,\gamma_{t+h})^2}{h^2}+\frac 1h\int_t^{t+h}v_{\beta_{t,h}}(s,s)^2\ud s\le \frac 1h\int_t^{t+h}h(s,\beta_{t,h}(s))^2\ud s.
\end{align*}
For fixed $t_0\in [a,b]$, the estimate
\begin{align*}
	\frac 1h\int_t^{t+h}v_{\beta_{t,h}}(s,s)^2\ud s\ge &\left(\frac 1h\int_t^{t+h}v_{\beta_{t,h}}(s,s)\ud s\right)^2\ge \frac{1}{(1+\omega(t-t_0))^2}\left(\frac 1h\int_t^{t+h}v_{\beta_{t,h}}(t_0,s)\ud s\right)^2\\
	\ge & \frac{1}{(1+\omega(t-t_0))^2}\left(\frac{d_{t_0}(\beta_t,\beta_{t+h})}{h}\right)^2
\end{align*}
implies (using a similar argument in case $h<0$) that
\begin{align*}
	\limsup_{h\to 0}\frac{\tau(\gamma_t,\gamma_{t+h})^2}{h^2}+\frac{v_\beta(t_0,t)^2}{(1+\omega(t-t_0))^2}\le h(t,\beta(t))^2\quad a.e.\ t.
\end{align*}
Choosing a countable dense set of $t_0\in [a,b]$ this easily implies that 
\begin{align*}
	\limsup_{h\to 0}\left(\frac{\tau(\gamma_t,\gamma_{t+h})}{h}\right)^2 \le h(t,\beta(t))^2-v_\beta(t,t)^2\quad a.e.\ t
\end{align*}
Since $L_\tau(\gamma|_{[a,t+h]})-L_\tau(\gamma|_{[a,t]})\le \tau(\gamma_t,\gamma_{t+h})$, it follows that the function $t\mapsto L_\tau(\gamma_{[a,t]})$ is absolutely continuous and its derivative squared is bounded a.e. from above by $h(t,\beta(t))^2-v_\beta(t,t)^2$. This implies that $L_\tau(\gamma)\le \tau(\gamma)$ and completes the proof. 
\end{proof}

The following result will be useful in proving that future and past lightcones are open under the assumption \eqref{eq:lip}. 

\begin{prop}\label{prop:strict-future-open}
	Suppose $\gamma_0:q_0\curvearrowright p_0$ is a future directed Lipschitz time-like curve with $|\tau_{\gamma_0}'|\ge c_0>0$ a.e. Then there exists $c>0$ and open neighbourhoods $U,V$ of $q_0$ and $p_0$, respectively such that for any $q\in U$ and $p\in V$ there exists a future directed time-like curve $\gamma:q\curvearrowright p$ with $|\tau_\gamma'|\ge c$.
\end{prop}

In the proof we need a technical lemma.

\begin{lemma}\label{lem:affine-perturbation-timelike}
	Let $J_0\subset I$ and $K\subset X$ be compact. Suppose $\gamma=(\alpha,\beta):J\to J_0\times K\subset hI\times_{\mathcal{F}}X$ is a future directed curve such that
	\begin{align*}
		\tau_\gamma'(t)^2=h(\alpha_t,\beta_t)\alpha'(t)^2-v_\beta(\alpha_t,t)^2\ge c_0^2>0
	\end{align*}
for a.e. $t\in J$, and
\begin{align*}
	\max\{\sup_{J_0\times K}h, |\alpha(J)|\inv, \|\alpha'\|_{L^\infty(J)},\|v_\beta\|_{L^\infty(J_0\times J)}\}\le L.
\end{align*}

Then there exists $\delta_0=\delta_0(L,J_0,K)$ with the following property. Whenever $J'\subset J_0$ is a compact subinterval with $|J'\triangle \alpha(J)|<\delta_0$, and $A:\alpha(J)\to J'$ is an increasing affine bijection, the perturbed curve $\gamma_A:=(A\circ\alpha,\beta)$ satisfies $\tau_{\gamma_A}'(t)\ge c_0/2$ a.e. $t\in J$.
\end{lemma}

\begin{proof}
	Let $\alpha(J)=[a,b]$ and $J'=[a',b']$. Then
	\[
	A(s)=a+\frac{b'-a'}{b-a}(s-a),\quad s\in [a,b]
	\]
	and $\delta:=|\alpha(J)\triangle J'|=|a-a'|+|b-b'|$. Observe that
	\begin{align*}
	\sup_{s\in \alpha(J)}|A(s)-s|\le 2\delta \quad\textrm{and}\quad |(A')^2-1|\le \frac{3\delta}{|\alpha(J)|}
	\end{align*}
as long as $\delta\le |\alpha(J)|$. By \eqref{eq:cont} and Lemma \ref{lem:lip_cont} have that
	\begin{align*}
		h(A(\alpha_t),\beta_t)^2& \ge h(\alpha_t,\beta_t)^2-\omega_{h^2}(2\delta)\\
		v_\beta(A(\alpha_t),t)^2 &\le [v_{\beta}(\alpha_t,t)+\|v_\beta\|_{L^\infty(J_0\times J)}\omega(2\delta)]^2\le v_{\beta}(\alpha_t,t)^2+3\|v_\beta\|_{L^\infty(J_0\times J)}^2\omega(2\delta),
	\end{align*}
where $\omega_{h^2}$ is the modulus of continuity of $h^2|_{J_0\times K}$. Denoting $\tilde\omega=\omega_{h^2}+\omega$ we have the estimate

\begin{align*}
	\tau_{\gamma_A}'(t)^2=&h(A(\alpha_t),\beta_t)^2(A')^2\alpha'(t)^2-v_\beta(A(\alpha_t),t)^2\\
	 \ge & (A')^2[h(\alpha_t,\beta_t)^2-\omega_{h^2}(2\delta)]\alpha'(t)^2-v_{\beta}(\alpha_t,t)^2-3\|v_\beta\|_{L^\infty(J_0\times J)}^2\omega(2\delta) \\
	\ge & h(\alpha_t,\beta_t)^2\alpha'(t)^2-\frac{3\delta}{|\alpha(J)|}\|\alpha'\|_{L^\infty(J)}^2(\sup_{J_0\times K}h)^2-4\|\alpha'\|_{L^\infty(J)}^2\tilde\omega(2\delta)\\
	&-v_{\beta}(\alpha_t,t)^2-3\|v_\beta\|_{L^\infty(J_0\times J)}^2\tilde\omega(2\delta)\\
	\ge &c_0^2-3L^5\delta-7L^2\tilde\omega(2\delta)
\end{align*}
	The proof is completed by choosing $\delta_0$ small enough such that $3L^5\delta-7L^2\tilde\omega(2\delta)<c_0^2/2$. 
\end{proof}

\begin{proof}[Proof of Proposition \ref{prop:strict-future-open}]
	Let $J_0\subset I$ and $K\subset X$ be compact sets so that $J_0\times K$ contains a neughbourhood of $\im(\beta)$. Write $\gamma_0=(\alpha_0,\beta_0):J\to hI\times_{\mathcal{F}}X$ and let $L$ be such that the metrics $\{d_s\}_{s\in J_0}$ are pairwise $L$-bi-Lipschitz on $K$, and
	\begin{align}\label{eq:L}
		(\inf_{J_0\times K}h)^{-2}+\sup_{J_0\times K}h+ |\alpha(J)|\inv+ \|\alpha'\|_{L^\infty(J)}+\|v_\beta\|_{L^\infty(J_0\times J)}\le L.
	\end{align}
Define
\[
U=(a_0-\delta,a_0+\delta)\times B_{d_a}(x_0,\delta),\quad V=(b_0-\delta,b_0+\delta)\times B_{d_b}(y_0,\delta), 
\]
where $\delta\le \frac{\delta_0}{2(L^2+1)}$ is small enough so that $U\cup V\subset J_0\times K$; here $\delta_0$ is the constant given by Lemma \ref{lem:affine-perturbation-timelike}.
Given $q=(a,x)\in U$ and $p=(b,y)\in V$, let
\[
\beta_x:[a,a+Ld_a(x,x_0)]\to X,\quad \beta_y:[b-Ld_b(y,y_0),b]\to X
\]
be constant speed parametrized geodesics $x\curvearrowright x_0$ and $y_0\curvearrowright y$, with respect to $d_a$ and $d_b$, respectively. It follows that 
\begin{align}\label{eq:b-x-y-speed}
	v_{\beta_x}(t,t)\le L|(\beta_x)_t'|_{d_a}=\frac 1L,\quad v_{\beta_y}(t,t)\le L|(\beta_y)_t'|_{d_b}\le \frac 1L 
\end{align}
a.e. $t$. By \eqref{eq:L} and \eqref{eq:b-x-y-speed} we have that
\begin{align*}
h(t,\beta_x(t))^2&-v_{\beta_x}(t,t)^2\ge \frac 1L-L^2\cdot\left(\frac 1L\right)^2=\frac 1L-\frac 1{L^2}, \quad a.e.\ t\in [a,a+L^2d_a(x_0,x)],\\
h(t,\beta_y(t))^2&-v_{\beta_y}(t,t)^2\ge \frac 1L-L^2\cdot\left(\frac 1L\right)^2=\frac 1L-\frac 1{L^2} \quad a.e.\ t\in [b-L^2d_b(y_0,y),y].
\end{align*}
Thus the curves $\gamma_x=(\id,\beta_x)$ and $\gamma_y=(\id,\beta_y)$ satisfy
\begin{align*}
	\tau_{\gamma_x}'\ge \sqrt{\frac 1L-\frac 1{L^2}},\quad \tau_{\gamma_y}'\ge \sqrt{\frac 1L-\frac 1{L^2}}\quad a.e.
\end{align*}
Oberve that $\gamma_x:q\curvearrowright \tilde q_0$ and $\gamma_y:\tilde p_0\curvearrowright p$, where
\[
q_0=(a+L^2d_a(x,x_0),x_0),\quad p_0=(b-L^2d_b(y_0,y),y_0).
\]
Choosing $J'=[a+L^2d_a(x,x_0),b-L^2d_b(y_0,y)]$, we note that
\[
|[a_0,b_0]\triangle J'| =|a+L^2d_a(x_0,x)-a_0|+|b-L^2d_b(y_0,y)-b_0|< 2\delta+2L^2\delta= \delta_0.
\]
Thus we may apply Lemma \ref{lem:affine-perturbation-timelike} to obtain that $\gamma_A:=(A\circ\alpha_0,\beta_0)$, where $A:[a_0,b_0]\to J'$ is the increasing affine bijection, satisfies $\tau_{\gamma_A}'\ge c_0/2$. Since $\tau_A:\tilde q_0\curvearrowright \tilde p_0$. The proof is now complete because the concatenation $\gamma:=\gamma_y\gamma_A\gamma_x:p\curvearrowright q$ is Lipschitz and satisfies
\[
\tau_\gamma'\ge \min\left\{\frac {c_0}2,\sqrt{\frac 1L-\frac 1{L^2}}\right\}\quad a.e.
\]
\end{proof}

\begin{remark}
	Define a strict time-like relation by setting $q\ll_sp$ if there exists a causal Lipschitz curve $\gamma:q\curvearrowright p$ such that $\tau_\gamma'\ge c$ for some $c>0$. Then Proposition \ref{prop:strict-future-open} shows that the corresponding future (and past) cones $I^\pm_s(p)$ are open \emph{assuming only \eqref{eq:cont}} (typically \eqref{eq:lip} is needed to guarantee openness of time-like future cones). In general Lorentzian length spaces $\ll_s$ need not coincide with $\ll$, and we do not know whether they can differ in our setting. In Section \ref{sec:lorentzian_length_space} we will show that $\ll=\ll_s$ in $hI\times_{\mathcal{F}}X$ assuming \eqref{eq:lip}.
\end{remark}

\subsection{Lorentzian manifolds}\label{sec:manifolds} Assume that $X=M$ is a smooth manifold, $h:I\times M\to (0,\infty)$ is continuous, and $\{g_s\}_{s\in I}$ is a collection of continuous Riemannian metrics on $M$ such that the local coordinate representations $(s,p)\mapsto (g_s)_{ij}(p)$ are continuous in $(s,p)$.
\begin{remark}
It is standard that the continuity assumption above is equivalent to the following: $(s,p)\mapsto g_s(\bm X_p,\bm Y_p)$ is continuous for any smooth vector fields $\bm X,\bm Y\in C^\infty(TM)$.
\end{remark}
\noindent It follows that
\begin{align}\label{eq:lor-prod-metric}
g\defeq -h^2\ud s^2+g_s
\end{align}
is a continuous Lorentzian metric on $I\times M$. Thus $(I\times M,g)$ is a continuous space-time and we denote by $(I\times M,\le_g,\ll_g,\tau_g)$ the induced Lorentzian length structure.

Now consider the family of length metrics $\mathcal F=\{ d_{g_s}\}_{s\in I}$ on $M$ associated to $\{g_s\}_{s\in I}$.

\begin{lemma}\label{lem:cont-cond-satisfied}
Under the hypotheses above, $\mathcal F$ satisfies \eqref{eq:cont}. Moreover, for any absolutely continuous curve $\beta:[a,b]\to X$ there exists a null set $N\subset [a,b]$ such that 
\begin{align}\label{eq:gen-speed-agree}
	v_\beta(s,t)^2=g_s(\beta_t',\beta_t'),\quad (s,t)\in I\times ([a,b]\setminus N).
\end{align}
\end{lemma}
\begin{proof}
Let $J\subset I$ and $K\subset M$ be compact, and let $\tilde g=\ud s^2+g_s$ be the Riemannian metric on $I\times M$. Set  $\tilde d_s=d_{\tilde g}|_{\{s\}\times X}$. Choose $r>0$ so that $\overline B_{d_{\tilde g}}(J\times K,r)$ is compact. By the continuity of the metric, there exists a modulus of continuity $\omega:[0,\infty)\to [0,\infty)$ such that $\sqrt{g_s(v,v)}\le\sqrt{g_{s'}(v,v)}+\omega(|s-s'|)\sqrt{g_{s'}(v,v)}$ for all $s,s'\in J_r\defeq pr_I(\overline B_{d_{\tilde g}}(J\times K,r))$ and $v\in T_pM$ with $p\in K_r\defeq pr_X(\overline B_{d_{\tilde g}}(J\times K,r))$. In particular if $x,y\in K$ and $\inf_{s\in J}d_{\tilde g}((s,x),(s,y))<r$, then $d_s(x,y)\le (1+\omega(|J_r|)))\tilde d_{s'}(x,y)$. Thus the metrics $d_s, \tilde d_{s'}$ are locally bi-Lipschitz equivalent for $s,s'\in J$ with uniform bi-Lipschitz constant $C=C(J,K)$. For $x,y\in K$ with $\inf_{s\in J}d_s(x,y)<r/C$, we have that any $d_s$-geodesic $\beta:x\curvearrowright y$ lies in $K$ for any $s\in J$, and thus
\begin{align*}
	d_{s'}(x,y)\le &\ell_{d_{s'}}(\beta)=\int_0^1\sqrt{g_{s'}(\beta_t',\beta_t')}\ud t\le \int_0^1(1+\omega(|s-s'|))\sqrt{g_{s}(\beta_t',\beta_t')}\ud t \\
	\le & (1+\omega(|s-s'|))d_s(x,y).
\end{align*}
This proves \eqref{eq:cont}.

To prove the second claim no that, since $\beta$ is absolutely continuous, $\beta_t'$ exists for a.e. $t$. Let $N\subset [a,b]$ be a null-set such that every $t\in [a,b]\setminus N$ is a Lebesgue point of $t\mapsto \beta_t'$. For fixed $s\in I$ we have that $|\beta_t'|_{d_{g_s}}^2=g_s(\beta_t',\beta_t')$ a.e. $t$. In particular for $(s,t)\in I\times ([a,b]\setminus N)$ we have that
\begin{align*}
	\sqrt{g_s(\beta_t',\beta_t')}=\lim_{h\to 0}\frac 1h\int_t^{t+h}\sqrt{ g_s(\beta_\tau',\beta_\tau')}\ud\tau=\limsup_{h\to 0}\frac 1h\int_t^{t+h}|\beta_\tau'|_{d_{g_s}}\ud\tau=v_\beta(s,t)
\end{align*}
\end{proof}

In particular, the length element of $g$ agrees with \eqref{eq:lorentzian-length-element} a.e. on absolutely continuous curves.
\begin{cor}\label{cor:length-element-compatible}
If $\gamma=(\alpha,\beta)$ is an absolutely continuous curve, we have that
\begin{align*}
	g(\gamma_t',\gamma_t')=-|\tau_\gamma'|^2(t)=-h(\gamma_t)\alpha'(t)^2+v_\beta(\alpha_t,t)^2\quad a.e.\ t.
\end{align*}
\end{cor}

We now give the proof of Corollary \ref{cor:compatible}.
\begin{proof}[Proof of Corollary \ref{cor:compatible}]
By Corollary \ref{cor:length-element-compatible} we have that a curve $\gamma$ in $I\times M$ is causal with respect to $(I\times M,\le_g,\ll_g,\tau_g)$ if and only if it is causal in $hI\times_\mathcal FM$. Moreover, in this case $\tau_g(\gamma)=\tau(\gamma)$ and thus $\tau_g=\tau$. This completes the proof.
\end{proof}

Next we verify \eqref{eq:lip} for Lipschitz metrics.

\begin{cor}\label{cor:lip-metric}
Suppose that $g$ in \eqref{eq:lor-prod-metric} is Lipschitz (i.e. $h$ and the local coordinate representations $(s,p)\mapsto (g_s)_{ij}(p))$ are locally Lipschitz). Then $\mathcal F=\{d_{g_s}\}_{s\in I}$ satisfies \eqref{eq:lip}.
\end{cor}
\begin{proof}
	Observe that, assuming $(s,p)\mapsto g_{ij}(p)$ is Lipschitz instead of merely continuous, we may choose $\omega(a)=Ca$ for some $C=C(J,K)$ in the proof of \eqref{eq:cont} in Lemma \ref{lem:cont-cond-satisfied}. This directly implies \eqref{eq:lip}.
\end{proof}

\section{Lorentzian length space structure of $hI\times_\mathcal FX$} \label{sec:lorentzian_length_space}

In this section we prove Theorems \ref{thm:push-up} and \ref{thm:glob_hyp} as well as Corollary \ref{cor:lor-len-sp}.

\subsection{Modifying causal curves via Caratheodory theory of ODE's}

Throughout this subsection we assume \eqref{eq:lip}. We prove Theorem \ref{thm:push-up} by reducing it to the two dimensional case and using Caratheodory theory of ODE's via the following proposition. We refer the reader to Appendix \ref{sec:ODE} for the results of Caratheodory theory of ODE's needed below.

\begin{prop}\label{prop:max-curve-modify}
Suppose $\gamma=(\alpha,\beta):q\curvearrowright p$ is a causal curve with $\tau(\gamma)>0$ defined on an interval $J$. Then there exists an absolutely continuous increasing function $y:J\to I$ with $y(J)=\alpha(J)$ such that
\begin{align*}
h(y_t,\beta_t)^2y'(t)^2-v_\beta(y_t,t)^2=c\in (0,\tau(q,p)]
\end{align*}
for all $t\in J$.
\end{prop}
\begin{remark}\label{rmk:obs_on_y}
	\begin{itemize}
		\item[(a)] In particular $(y,\beta)$ is a time-like curve with the same start and end-points as $\gamma$, and 
		\item[(b)] the Lipschitz constant of $y$ is at most $\displaystyle \frac{\|v_\beta\|_{L^\infty(\alpha(J)\times J)}+\tau(\gamma)}{\underset{\alpha(J)\times\beta(J)}{\inf} f}$ (which can be $+\infty$ if $\beta$ is not Lipschitz).
	\end{itemize}
\end{remark}
\begin{proof}
	Without loss of generality we may assume that $J=[0,1]$. Denote $\gamma_0=(a,x)$ and $\gamma_1=(b,y)$. Given $\eps \in [0,\tau(\gamma)]$ we define $\Phi_\eps:I\times [0,1]\to \R$ by 
	\begin{align*}
		\Phi_\eps(s,t)=\frac{(v_\beta(s,t)^2+\eps^2)^{1/2}}{h(s,\beta_t)}.
	\end{align*}

By \eqref{eq:lip}, \eqref{eq:lip-h-time-var} and Lemma \ref{lem:lip_cont} there exists $L\in L^1(0,1)$ such that 
\begin{align*}
|\Phi_\eps(s,t)-\Phi_\eps(s',t)| \le L(t)|s-s'|,\quad s,s'\in \im(\alpha),
\end{align*}
i.e. $\Phi_\eps$ is Caratheodory-Lipschitz (uniformly in $\eps\in [0,\tau(\gamma)]$). Theorem \ref{thm:car-ODE} implies that the ODE
\begin{align*}
	\left\{ 
	\begin{array}{l}
		y'=\Phi_\eps(y,t) \\
		y(0)=a
	\end{array}
	\right.
\end{align*}
admits a unique absolutely continuous solution $y_\eps$ which varies continuously with $\eps$. For each $\eps$ the curve $\gamma_\eps=(y_\eps,\beta)$ satisfies

\begin{align*}
	h(y_\eps(t),\beta_t)^2y_\eps'(t)^2-v_\beta(y_\eps(t),t)^2=\eps^2
\end{align*}
and $\tau(\gamma_\eps)=\eps$. We show that
\begin{align}\label{eq:vary-e}
	y_0(1)<b\quad\textrm{and}\quad y_{\tau(q,p)}(1)\ge b.
\end{align}

Since
\begin{align*}
	y_0'-\frac{v_\beta(y_0(t),t)}{h(y_0(t),\beta_t)}=y_0'-\Phi_0(y_0,t)=0\le \alpha'_t-\Phi_0(\alpha,t),
\end{align*}
the comparison principle for ODE's (Theorem \ref{thm:car-comparison}) there exists $c\in [0,1]$ such that $y_0(t)=\alpha(t)$ whenever $t\in [0,c]$ and $y_0(t)<\alpha(t)$ for $t\in (c,1]$. If $c=1$, then $y_0=\alpha$ which implies $\gamma=\gamma_0$, contradicting the fact that $\tau(\gamma)>0$. Thus we have that $y_0(1)<\alpha(1)=b$. To finish the proof of \eqref{eq:vary-e} suppose that $y_{\tau(q,p)}(1)<b$. Then the concatenation $\tilde\gamma$ of $\gamma_{\tau(q,p)}$ and the vertical segment $[y_{\tau(q,p)}(1),b]\times \{y\}$ is a time-like curve connecting $q$ and $p$ with Lorentzian length $\tau(\tilde\gamma)=\tau(\gamma_{\tau(q,p)})+(b-y_{\tau(q,p)})>\tau(q,p)$, which is a contradiction.

We conclude the proof of the proposition using \eqref{eq:vary-e}. Since $y_\varepsilon(1)$ varies continuously in $\eps$, \eqref{eq:vary-e} implies that there exists $\eps_b\in (0,\tau(q,p)]$ for which $y_{\eps_b}(1)=b$. Thus $y=y_{\eps_b}$ is the required function.
\end{proof}

\subsection{$hI\times_{\mathcal{F}}X$ is a Lorentzian length space}

\begin{proof}[Proof of Theorem \ref{thm:push-up}]
Suppose $q\le p\le r$. By the reverse triangle inequality (cf. Remark \ref{rmk:pre_pre_lorentz_length_sp}(2)) $\tau(q,r)\ge \tau(q,p)+\tau(p,r)$. If $q\ll p$ or $p\ll r$, one of the summands is positive and thus $\tau(q,r)>0$. It follows that there exists a causal curve $\gamma:q\curvearrowright r$ with $\tau(\gamma)>0$. Proposition \ref{prop:max-curve-modify} yields the existence of a time-like curve $p\curvearrowright r$, showing that $p\ll r$. This completes the proof.
\end{proof}

Next we focus on Corollary \ref{cor:lor-len-sp}. To establish it we need to show the lower semicontinuity of $\tau$. This, in turn will be an easy consequence of the openness of the relation $\ll$. 

\begin{prop}\label{prop:open-future}
Suppose $q_0,p_0\in hI\times_{\mathcal F}X$, $q_0\ll p_0$. Then there exist neighbourhoods $U$ and $V$ of $q_0$ and $p_0$, respectively, such that $q\ll p$ whenever $q\in U$ and $p\in V$.
\end{prop}
In particular, $I^\pm(p)$ is open for any $p\in hI\times_{\mathcal{F}}X$. 

\begin{proof}
The claim follows from Proposition \ref{prop:strict-future-open} once we establish the existence of a Lipschtiz curve $\gamma_0:q_0\curvearrowright p_0$ with $\tau_{\gamma_0}'\ge c$ for some $c>0$. To this end let $\gamma=(\alpha,\beta):q_0\curvearrowright p_0$ be a causal curve with $\tau(\gamma)>0$ parametrized so that $\alpha$ and $\beta$ are Lipschitz. Then $\gamma_0=(y,\beta)$, where $y$ is given by Proposition \ref{prop:max-curve-modify}, satisfies the required properties, and thus the proof is complete.
\end{proof}

\begin{proof}[Proof of Corollary \ref{cor:lor-len-sp}]
We have that $(hI\times_{\mathcal{F}}X,\le,\ll)$ is a causal set which satisfies the push-up property (Theorem \ref{thm:push-up}). Moreover $\tau$ satisfies the reverse triangle inequality (Remark \ref{rmk:pre_pre_lorentz_length_sp}).

We use the argument in the proof of \cite[Lemma 3.25]{agks21} to show that $\tau$ is lower semicontinuous. Let $p,q\in hI\times_{\mathcal F}X$, $\eps>0$ be arbitrary and assume without loss of generality that $\tau(p,q)>0$. Let $\gamma:[a,b]\to hI\times_{\mathcal F}X$ be a causal curve $p\curvearrowright q$ with $\tau(\gamma)>0$ such that $\tau(\gamma)>\tau(q,p)-\eps/3$. We may find $a<t_1\le t_2<b $ such that $0<\tau(\gamma|_{[a,t_1]}),\tau(\gamma_{[t_2,b]})<\eps/3$. Setting $U\defeq I^-(\gamma(t_1))$ and $V\defeq I^+(\gamma(t_2))$ we have that $U$ and $V$ are open and $(p,q)\in U\times V$, cf. Proposition \ref{prop:open-future}. For $(\tilde p,\tilde q)\in U\times V$ there are timelike curves $\tilde p\curvearrowright \gamma(t_1)$ and $\gamma(t_2)\curvearrowright \tilde q$ and their concatenation with $\gamma|_{[t_1,t_2]}$ is a causal curve $\tilde p\curvearrowright \tilde q$. Thus
\begin{align*}
	\tau(\tilde p,\tilde q)\ge \tau(\gamma|_{[t_1,t_2]})>\tau(\gamma)-\eps/2-\eps/2=\tau(p,q)-\eps.
\end{align*}
This proves the claimed lower semicontinuity. Thus we have that $hI\times_{\mathcal F}X$ is a causally path connected Lorentzian pre-length space. Proposition \ref{prop:causality_stable} and Corollary \ref{cor:loc-comp-causal-diamond} implies that $hI\times_{\mathcal F}X$ is locally causally closed. For any $p=(a,x)\in hI\times_\mathcal FX$ there exists $h>0$ such that $[a-h,a+h]\subset I$, and the chronological diamond $I_p=((a-h,x),(a+h,x))$ is a localizing neighbourhood of $p$, cf. Corollary \ref{cor:loc-comp-causal-diamond}. Finally, Proposition \ref{prop:lor=var} implies that $\tau(q,p)=\sup_{\gamma:q\curvearrowright p} L_\tau(\gamma)$, so that it is a Lorentzian length space.
\end{proof}

Before we consider the global hyperbolicity of the generalized Lorentzian product, we record here the fact that it is always strongly causal.

\begin{prop}\label{prop:strongly-local}
$hI\times_\mathcal FX$ strongly causal: chronological diamonds $I(p,q)$ form a basis of the topology of $I\times X$.
\end{prop}
\begin{proof}
Every neighbourhood $U$ of a point $p=(a,x)$ contains the chronological diamond $I((a-h,x),(a+h,x))\ni p$ for small enough $h>0$ by Corollary \ref{cor:loc-comp-causal-diamond} and Remark \ref{rmk:causal-curve-length-bd}. Since chronological diamonds are open, it follows that they form a basis of the topology of $hI_\mathcal FX$. 
\end{proof}

\subsection{Global hyperbolicity of $hI\times_{\mathcal{F}}X$} \label{sec:glob_hyp}
Recall the metric $(d_{\mathcal F,h})_{h\inv}$ on $I\times X$ from Section \ref{sec:gen-prod-metr}. Global hyperbolicity is connected to the properness of this metric.

\begin{lemma}\label{lem:cpt-diamonds}
Let $J\subset I$ be compact and denote $\mathcal F_J=\{d_s\}_{s\in J}$. If $(J\times X,(d_{\mathcal F_J,h})_{h\inv})$ is proper, then every causal diamond $J(p,q)\subset hI\times_\mathcal FX$ with $p,q\in J\times X$ is compact.
\end{lemma}
\begin{proof}
	
Let $p=(a,x)$ and $q=(b,y)\in J\times X$. Since $(J\times X,(d_{\mathcal F_J,h})_{h\inv})$ is proper there exists a compact set $K\subset X$ such that any absolutely continuous curve starting at $p$ with length at most $\sqrt 2(b-a)$ lies in $J\times K$. By Remark \ref{rmk:causal-curve-length-bd} every causal curve $\gamma:p\curvearrowright q'$ with $q'\in J(p,q)$ has length at most $\sqrt 2(b-a)$, and thus it follows that $J(p,q)\subset J\times K$. By Corollary \ref{cor:loc-comp-causal-diamond}(a) $J(p,q)$ is compact.
\end{proof}

Since $d_{\mathcal F_J,h}$ is a locally compact length metric on $J\times X$, Lemma \ref{lem:properness-criterion} immediately yields the following condition for properness.

\begin{cor}\label{cor:properness-criterion}
	Suppose $h:I\times X\to (0,\infty)$ is continuous and $\mathcal F$ satisfies \eqref{eq:cont}. If $J\subset I$ is compact and
	\begin{equation}\label{eq:properness-criterion}
	\int_0^\infty\frac{v_\beta(\alpha_t,t)}{h(\gamma_t)}\ud t=\infty
	\end{equation}
whenever $\gamma=(\alpha,\beta):[0,\infty)\to J\times X$ is a locally absolutely continuous curve tending to infinity, then $(J\times X,(d_{\mathcal F_J,h})_{h\inv})$ is a proper metric space.
\end{cor}

\begin{proof}[Proof of Theorem \ref{thm:glob_hyp}]
It follows from Remark \ref{rmk:causal-curve-length-bd} that $hI\times_\mathcal FX$ is totally non-imprisoning. By \eqref{eq:inf-length} and Corollary \ref{cor:properness-criterion} we have that $(J\times X,(d_{\mathcal F_J,h})_{h\inv})$ is proper for each compact subinterval $J\subset I$. By Lemma \ref{lem:cpt-diamonds} causal diamonds in $hI\times_\mathcal FX$ are compact, and thus $hI\times_\mathcal FX$ is globally hyperbolic.
\end{proof}

\subsection{Regularity of conformal generalized Lorentzian products}\label{sec:regular}
Throughout this subsection we assume that $d$ is a given locally compact length metric on $X$, and that $h,\rho:I\times X\to (0,\infty)$ are locally Lipschitz functions. For each $s\in I$, consider the conformally weighted length metric $d_s\defeq d_{\rho(s,\cdot)}$ on $X$, satisfying 
\begin{align*}
	|\beta_t'|_{d_s}=\rho(s,\beta_t)|\beta_t'|_d\quad a.e.\ t
\end{align*}
for every $\beta\in AC(X,d)$. Set $\mathcal F=\{d_s\}_{s\in I}$. It follows from \eqref{eq:rho-length} that
\begin{align}\label{eq:conformal-gen-speed}
v_\beta(s,t)=\rho(s,\beta_t)v_\beta(t),\quad (s,t)\in I\times [a,b]
\end{align}
for every absolutely continuous $\beta:[a,b]\to X$. Here
\[
v_\beta(t)=\limsup_{h\to 0}\frac 1h\int_t^{t+h}|\beta_u'|_d\,\ud u,\quad t\in [a,b].
\]
\begin{lemma}
The family $\mathcal F$ defined above satisfies \eqref{eq:lip}.
\end{lemma}
\begin{proof}
Let $J\subset I$ and $K\subset X$ be compact. Let $r_0=r_0(K)>0$ be such that $K'=\overline B_{d}(K,r_0)$ is compact, and set $A=A(J,K)=\inf_{J\times K'}\rho$. Observe that, if $s\in J$ and $\beta$ is a curve in $X$ such that $\ell_{d_s}(\beta)<Ar_0$, then $\beta$ lies in $K'$. 

Now suppose $x,y\in K$ and $\displaystyle \inf_{s\in J}d_s(x,y)<Ar_0$. Then there is $s_0\in J$ such that $d_{s_0}(x,y)<Ar_0$. By the observation above any sequence  $(\beta_j)$ of constant speed curves $x\curvearrowright y$ with $\ell_{d_{s_0}}(\beta_j)\to d_{s_0}(x,y)$ eventually lies in $K'$ and thus subconverges to a $d_s$-geodesic $\beta:x\curvearrowright y$ lying in $K'$. It follows that
\begin{align*}
d_{s}(x,y)\le &\int\rho(s,\beta_t)|\beta_t'|\ud t\le d_{s_0}(x,y)+\int|\rho(s,\beta_t)-\rho(s_0,\beta_t)||\beta_t'|\ud t \\
\le &d_{s_0}(x,y)\Big(1+\frac{\LIP(\rho|_{J\times K'})}{A}|s-s_0|\Big)
\end{align*}
for any $s\in J$. Set $C=\frac{\LIP(\rho|_{J\times K'})}{A}$ and $r=Ar_0/(1+C|J|)$. In particular if $\displaystyle \inf_{s\in J}d_s(x,y)<r_0$, then $d_s(x,y)<Ar_0$ for all $s\in J$, and the argument above yields 
\[
\frac{d_{s'}(x,y)}{d_s(x,y)}\le 1+C|s-s'|,\quad x,y\in K,\ s,s'\in J,
\]
proving \eqref{eq:lip}
\end{proof}

We prove Theorem \ref{thm:regular-loc-lor-len-sp} as a corollary of the following result.

\begin{thm}\label{thm:causal-character}
Let $\gamma:[a,b]\to hI\times_{\mathcal{F}}X$ be a maximizing causal curve with $L_\tau(\gamma)>0$. Then $\gamma$ is time-like, i.e. $\tau_\gamma'(t)>0$ a.e. $t\in [a,b]$.
\end{thm}

Since reparametrization does not change causal character, we may assume that $\gamma=(\alpha,\beta)$ parametrized such that $|\beta_t'|_d=c$ a.e. for some constant $c>0$, and $\alpha$ is Lipschitz. Consider the space $Q=\alpha([a,b])\times [a,b]\subset \R^2$ equipped with with the Lorentzian metric
\begin{align*}
	g_Q=-h(s,\beta_{t})^2\ud s^2+v_\beta(s,t)^2\ud t^2.
\end{align*}
Note that $v_\beta(s,t)=c\rho(s,\beta_t)$, cf. \eqref{eq:conformal-gen-speed}.
\begin{lemma}
	The Lorentzian metric $g_Q$ is Lipschitz.
\end{lemma}
\begin{proof}
Denote $F(s,t)=h(s,\beta(t))^2$. Since $\beta$ and $h^2|_{\alpha([a,b])\times \beta([a,b])}$ are Lipschitz, it follows that $-F$ is Lipschitz. Similarly $(s,t)\mapsto v_\beta(s,t)^2=c^2\rho(s,\beta_t)^2$ is Lipschitz.
\end{proof}

\begin{lemma}\label{lem:Q-causal-iff-causal}
	A future directed curve $y=(y_0,y_1)$ in $Q$ is causal with respect to $g_Q$ if and only if $\gamma_y\defeq (y_0,\beta\circ y_1)$ is causal in $hI\times_{\mathcal{F}}X$. In this case we have 
	\begin{align*}
		L_{g_Q}(y)=\tau(\gamma_y).
	\end{align*}
\end{lemma}
\begin{proof}
	By Lemma \ref{lem:reparam} we have that
	\begin{align*}
		g_Q(y_t',y_t')=&-h(y_0(t),\beta(y_1(t)))^2y_0'(t)^2+v_\beta(y_0(t),y_1(t))^2y_1'(t)^2\\
		=&-h(y_0(t),\beta\circ y_1(t))^2y_0'(t)^2+v_{\beta\circ y_1}(y_0(t),t)^2=-\tau_{\gamma_y}'(t)^2
	\end{align*}
a.e. $t$ from which the first claim immediately follows. It also follows that, if $y$ is causal, then
\begin{align*}
	L_{g_Q}(y)=\int \sqrt{-g_Q(y_t',y_t')}\ud t=\int \tau_{\gamma_y}'(t)\ud t=\tau(\gamma_y).
\end{align*}
\end{proof}

\begin{cor}\label{cor:Q-max-curve}
For any $t,t'\in [a,b]$ with $t\le t'$ we have that $x:[t,t']\to Q$, $x(\lambda)=(\alpha_\lambda,\lambda)$ is a maximizing curve, and
\begin{align*}
\tau_Q(x(t),x(t'))=\tau(\gamma|_{[t,t']}).
\end{align*}
\end{cor}
\begin{proof}
	We use the notation of Lemma \ref{lem:Q-causal-iff-causal}. Note that $\gamma_x=\gamma|_{[t,t']}$ and thus $L_{g_Q}(x)=\tau(\gamma|_{[t,t']})$. If $x$ were not maximal, there would exist a causal curve $y:[t,t']\to Q$ with the same end points and $\tau(\gamma_y)=L_{g_Q}(y)>L_{g_Q}(x)=\tau(\gamma|_{[t,t']})$. Since $\gamma_y$ is a causal curve with the same endpoints as $\gamma|_{[t,t']}$, this would contradict the maximality of $\gamma$. 
\end{proof}

\begin{proof}[Proof of Theorem \ref{thm:causal-character}]
Consider the Lorentzian manifold $(Q,g_Q)$ and note that $x=(\alpha,\id):[a,b]\to Q$ is a maximizing curve $(\alpha_a,a)\curvearrowright(\alpha_b,b)$, cf. Corollary \ref{cor:Q-max-curve}. Recall that $\gamma_x=\gamma$ (using the notation of Lemma \ref{lem:Q-causal-iff-causal}). Since $L_{g_Q}(x)=\tau(\gamma)>0$ and $g_Q$ is Lipschitz, it follows that $(\alpha_a,a)\ll_Q(\alpha_b,b)$ and that $x$ is time-like, cf. [Lorentz-meets-Lipschitz, Prop 1.2]. In particular 
\begin{align*}
	0<-g_Q(x_t',x_t')=h(\alpha_t,\beta_t)^2\alpha'(t)^2-v_\beta(\alpha_t,t)^2\quad a.e.\ t,
\end{align*}
showing that $\gamma=(\alpha,\beta)$ is time-like. 
\end{proof}

\begin{proof}[Proof of Theorem \ref{thm:regular-loc-lor-len-sp}]
Since $hI\times_\mathcal F X$ is a Lorentzian length space the only remaining claim to establish is that every point has a localizing neighbourhood where maximizing curves with positive $\tau$-length are time-like. Indeed, given $p=(a,x)\in I\times X$, the localizing neighbourhood $U_p=I((a-h,x),(a+h,x))$ (for small enough $h>0$) satisfies the locally regularizable condition (iv) in Section \ref{sec:preli-lorentz} by Theorem \ref{thm:causal-character}. This completes the proof.
\end{proof}

\section{Synthetic curvature bounds}\label{sec:ricci}
We apply the synthetic time-like Ricci curvature bounds defined in Section \ref{sec:curv} to the Lorentzian length structure $hI\times_\mathcal FX$, where $\mathcal F$ satisfies \eqref{eq:cont}, and $h:I\to (0,\infty)$ is a continuous function of one variable.

Let $\mathcal M=\{m_s\}_{s\in I}$ be a family of measures as in Section \ref{sec:intro-TCD} satisfying \eqref{eq:measure-comp}, and consider the measure $\bm m=m_s\otimes h\ud s$ defined on the generalized Lorentzian product $hI\times_\mathcal FX$ by

\begin{align*}
	\int F(s,x)\ud \bm m(s,x)=\int_Ih(s)\int_X F(s,x)\ud m_s(x)\ud s
\end{align*}
for Borel functions $F:I\times X\to [0,+\infty]$. Let $G:I\times X\to (0,\infty)$, with $G$ and  $G\inv$ locally bounded, be such that $\displaystyle \int G\ud\bm m=1$, and define $m=m_G\in \mathcal P(X)$ by
\begin{align*}
	m(E)=\int_{I\times E}G\ud \bm m,\quad E\subset X.
\end{align*}
Let $g$ be the Radon--Nikodym derivative of $\bm m$ with respect to $h\mathcal L^1\otimes m$, i.e. $\bm m =g\cdot h\mathcal L^1\otimes m$.
Using \eqref{eq:measure-comp} we see that, for each compact $J\subset I$ and $K\subset X$, there exists a constant $C=C(J,K)>0$ with
\begin{align}\label{eq:density-bound}
	C\inv\le g\le C\quad \bm m-a.e.\textrm{ on } J\times K.
\end{align}
In what follows we choose a Borel representative of $g$ such that this holds pointwise everywhere for all $J,K$.
\begin{remark}
For each compact $J\subset I$ and $K\subset X$ there exists $C=C(J,K)>0$ such that $C\inv m_s|_K\le m|_K\le C m_s|_K$ for every $s\in J$.
\end{remark}

In the next theorem we denote by $|\cdot|_h$ the metric on $I$ given by  
\[
|t-s|_h=\int_s^th(\tau)\ud \tau,\quad s\le t.
\]

\begin{thm}\label{thm:k,n-convex}
	Suppose $\mathcal F$ and $\mathcal M$ satisfy \eqref{eq:cont} and \eqref{eq:measure-comp}, respectively, $h:I\to (0,\infty)$ is continuous, and assume that $(hI\times_{\mathcal{F}}X,m)$ satisfies the wTCD$_p$($K,N$)-condition. For any compact interval $J\subset I$ we have that, for $m$-a.e. $x\in X$, $g(\cdot,x)$ has a (continuous) $\mathcal H^1_{|\cdot|_h}$-representative  $g_{x}$ on $J$ that satisfies
	\begin{align}\label{eq:log-KN-convex}
		g_{x}(a_t)^{1/N}\ge \sigma_{K/N}^{1-t}(|b-a|_h)g_{x}(a)^{1/N}+\sigma_{K/N}^t(|b-a|_h)g_{x}(b)^{1/N},\quad a,b\in J,\ t\in [0,1],
	\end{align}
	where $a_t:[0,1]\to J$ is the $|\cdot|_h$-geodesic $a\curvearrowright b$.
\end{thm}

Observe that $\mathcal H^1_{|\cdot|_h}=h\mathcal L^1$. Moreover, denoting $L=\|h\|_{L^1(J)}=\mathcal H^1_{|\cdot|_h}(J)$, the function 
\begin{align*}
	H:J\to [0,L],\quad H(t)\defeq\int_{J\cap (-\infty,t]}h\ud \tau,
\end{align*}
defines an isometric bijection $(J,|\cdot|_h)\to ([0,L],|\cdot|)$ with $H_\ast(g\mathcal H_{|\cdot|_h}^1)=g\circ H\inv\mathcal L^1$. Since in the 1-dimensional situation $([0,L],|\cdot|, g\mathcal L^1)$, the RCD($K,N$)-condition is characterized by \eqref{eq:log-KN-convex} (see \cite[Theorem 1.1]{kit-lak16}), Theorem \ref{thm:k,n-convex} has the following immediate corollary.

\begin{cor}\label{cor:RCD-lines}
	Under the hypotheses of Theorem \ref{thm:k,n-convex}, for $m$-a.e. $x\in X$, we have that\newline $J_x=(J,|\cdot|_h,g(\cdot,x)\mathcal H^1_{|\cdot|_h})$ is an RCD($K,N$)-space for every compact interval $J\subset I$.
\end{cor}

We set up notation for the proof of Proposition \ref{thm:k,n-convex}. Let $\nu_0,\nu_1\in \mathcal P(J)$ and let $t\mapsto \nu_t$ be the (unique) geodesic with respect to the Wasserstein  distance
\begin{align}\label{eq:wass-dist}
W_{h}(\nu,\sigma)=\inf\left\{ \left(\frac 12\int_{J\times J} |x-y|_h^2\ud\bm\pi(x,y)\right)^{1/2}:\ \bm\pi \in \Pi(\nu,\sigma) \right\}.
\end{align}
Recall that $t\mapsto \nu_t$ is represented by an optimal dynamical plan $\bm\eta\in \mathcal P(Geo(J,|\cdot|_h))$, i.e $\nu_t=e_{t\ast}\bm\eta$, and that the optimal pairing in \eqref{eq:wass-dist} is given by $\bm\pi=(e_0,e_1)_\ast\bm\eta$. By \cite[Theorem 2.9]{sant15} $\bm\pi$ (and thus $\bm\eta$) is optimal for any cost $c(x,y)=f(|y-x|_h)$ with $f:\R\to\R$ strictly convex. Therefore $C\defeq \spt\bm\pi$ is $c$-cyclically monotone for any such $c$. Choosing $f(t)=L^p-|t|^p$ for $p\in (0,1)$ we obtain the following \emph{reverse} inequality in cyclic monotonicity.
\begin{align}\label{eq:cyclic-monotone}
	\sum_i^k|t_i-s_i|_h^p\ge \sum_i^k|t_i-s_{\sigma(i)}|_h^p,\quad p\in (0,1)
\end{align}
for all $(t_1,s_1),\ldots,(t_k,s_k)\in C$ and every permutation $\sigma$ of $\{1,\ldots, k\}$.
\begin{remark}\label{rmk:eta-max-concave-cost}
By the same argument we have that
\begin{align*}
\int|\alpha_1-\alpha_0|^p\ud\bm\eta\ge \int|\alpha_1-\alpha_0|^p\ud\tilde{\bm\eta}
\end{align*}
for any $p\in (0,1)$ and any $\tilde{\bm\eta}\in \mathcal P(C([0,1],J))$ with $e_{i\ast}\tilde{\bm\eta}= \nu_i, \ i=0,1$, with equality if and only if $\tilde{\bm\eta}=\bm\eta$.
\end{remark}

\begin{lemma}\label{lem:unique-optimal plan}
Let $\nu_0,\nu_1\in \mathcal P^{ac}(J,\mathcal H^1_{|\cdot|_h})$ be such that $\bm\eta\in {\rm OptGeo}(\nu_0,\nu_1)$ is concentrated on increasing curves, and let $B\subset X$ be a bounded Borel set. Define 
\begin{align*}
\mu_t^B\defeq \nu_t\times \frac{m|_B}{m(B)}.
\end{align*}
Then $(\mu_0^B,\mu_1^B)\in \mathcal P^{ac}(hI\times_{\mathcal{F}}X,\bm m)$ is time-like $p$-dualizable and $t\mapsto \mu_t^B$ is the unique $\ell_p$-geodesic between $\mu_0^B$ and $\mu_1^B$. Moreover
\begin{align}\label{eq:p-optimal-length}
	\ell_p(\mu_0^B,\mu_1^B)=\left(\int|\alpha_1-\alpha_0|_h^p\ud\bm\eta(\alpha)\right)^{1/p}.
\end{align}
\end{lemma}
In particular $(\mu_0^B,\mu_1^B)$ is strongly time-like $p$-dualizable.
\begin{proof}
	The curve $t\mapsto \mu_t^B$ is represented by $\bm\eta^B\in \mathcal P({\rm OpT}(hI\times_{\mathcal{F}}X))$, defined by
	\[
	\int F(\gamma)\ud\bm\eta^B(\gamma)\defeq \dashint_B F((\alpha,x))\ud\bm\eta(\alpha)\ud m(x),\quad F:C([0,1];I\times X)\to [0,\infty].
	\]
	Consequently $\bm\pi^B\defeq (e_0,e_1)_\ast\bm\eta^B\in \Pi_\ll(\mu_0,\mu_1)$. We show that $C:=\spt \bm\pi^B$ is $\ell^p$-cyclically monotone. Indeed, if $(p_j,q_j)\in C$ for $j=1,\ldots, N$ then $p_j=(t_j,x_j)$ and $q_j=(s_j,x_j)$ (so that $p_j$ and $q_j$ have the same $x$-component) and we have the estimate
	\begin{align*}
		\sum_j^N\ell^p(p_j,q_{\sigma(j)})\le \sum_j^N|t_j-s_{\sigma(j)}|_h^p\le \sum_j^N|t_j-s_j|_h^p = \sum_j^N\ell^p(p_j,q_{j})
	\end{align*}
for any permutation $\sigma$ by \eqref{eq:cyclic-monotone}. This shows that $\bm\pi^B$ is optimal, and a direct calculation yields \eqref{eq:p-optimal-length}.

To show uniqueness suppose $\bm\pi'\in \Pi_\ll(\mu_0^B,\mu_1^B)$ is an optimal coupling between $\mu_0^B$ and $\mu_1^B$, and let $\bm\eta'\in \mathcal P({\rm OpT}( hI\times_{\mathcal{F}}X ))$ represent $\bm\pi'$. By Remark \ref{rmk:eta-max-concave-cost} we have that

\begin{align*}
\ell_p(\mu_0^B,\mu_1^B)^p=\int L_\tau(\gamma)^p\ud\bm\eta'\le \int |p_1(\gamma_1)-p_1(\gamma_0)|_h^p\ud\bm\eta'=\int |\alpha_1-\alpha_0|_h^p\ud p_{1\ast}\bm\eta'\le \int|\alpha_1-\alpha_0|_h^p\ud\bm\eta.
\end{align*}
Since by \eqref{eq:p-optimal-length} the last inequality is an equality, Remark \ref{rmk:eta-max-concave-cost} implies that $p_{1\ast}\bm\eta'=\bm\eta$. In particular $\spt\bm\pi'=\spt\bm\pi^B$. Since for every $(p,q)\in\spt\bm\pi^B$ there is a \emph{unique} time-like geodesic $p\curvearrowright q$ (the vertical line with constant $X$-component), it follows that $\bm\eta'=\bm\eta$.
\end{proof}

\begin{remark}
	In particular ${\rm OpT}_p(\mu_0^B,\mu_1^B)=\{\bm\eta^B\}$.
\end{remark}

In the notation of Lemma \ref{lem:unique-optimal plan}, define 
\begin{align*}
	e_B(t)\defeq {\rm Ent}_{\bm m}(\mu_t^B),\quad e_x(t)\defeq{\rm Ent}_{\lambda_x}(\nu_t),\quad t\in [0,1],
\end{align*}
where $\lambda_x=g(\cdot,x)\mathcal H^1_{|\cdot|_h}$.
\begin{lemma}\label{lem:e-vs-e_x}
	For $m$-a.e. $x\in X$ we have that
	\begin{align}\label{eq:entropy}
	e_B(t)=\dashint_B e_x(t)\ud m-\log m(B).
	\end{align}
\end{lemma}
\begin{proof}
Let $\{\tilde\rho_t\}_{t\in [0,1]}$ be Borel functions on $J$ such that $\nu_t=\tilde \rho_t\mathcal H^1_{|\cdot|_h}$ and set
\begin{align*}
	\rho_t(s,x)\defeq \frac{\tilde\rho_t(s)}{g(s,x)},
\end{align*}
so that $\nu_t=\rho_t(\cdot,x)\lambda_x$ for $m$-a.e. $x\in X$. Since $\bm m=g\mathcal H^1_{|\cdot |_h}\otimes m$, we have that \newline $\displaystyle 	\mu_t^B=\frac{\chi_{J\times B}}{m(B)}\rho_t \bm m$. Thus
\begin{align*}
	e_B(t)=&\int_{J\times B}\frac{\rho_t}{m(B)}\log \frac{\rho_t}{m_J(B)} \ud\bm m\nonumber\\
	=& \dashint_B\int_J\rho_t\log \rho_t \ud\lambda_x\ud m(x)-\dashint_B\int_J\rho_t\log m(B)\ud\lambda_x\ud m(x)\nonumber\\
	=& \dashint_Be_x(t)\ud m(x)-\log m(B),
\end{align*}
completing the proof.
\end{proof}

\bigskip\noindent Assume now that $hI\times_{\mathcal F}X$ satisfies the wTCD$_p$($K,N$)-condition for some $K\in \R$, $N\ge 1$ and $p\in (0,1)$. It follows from the discussion above that $e_B$ satisfies
\begin{align}\label{eq:e_B-concavity}
	e^{-e_B(t)/N}\ge \sigma_{K/N}^{1-t}(D)e^{-e_B(0)/N}+\sigma_{K/N}^t(D)e^{-e_B(1)/N},\quad t\in [0,1],
\end{align}
where
\begin{align*}
	D\defeq \|\tau\|_{L^2(\bm\pi_B)}=
	\left(\dashint_B \int|\alpha_1-\alpha_0|_h^2\ud\bm\eta(\alpha)\ud m(x)\right)^{1/2}=W_{h}(\nu_0,\nu_2).
\end{align*}

\begin{lemma}\label{lem:e_x-KN-convex}
	Let $\nu_0,\nu_1\in \mathcal P^{ac}(J,\mathcal H^1_{|\cdot|_f})$ be as in Lemma \ref{lem:unique-optimal plan}. For $m$-a.e. $x\in X$, the function $e_x:[0,1]\to \R$ satisfies
	\begin{align}\label{eq:KN-convexity of e_x}
		e^{-e_x(t)/N}\ge \sigma_{K/N}^{1-t}(D)e^{-e_x(0)/N}+\sigma_{K/N}^{t}(D)e^{-e_x(1)/N},\quad t\in [0,1].
	\end{align}
\end{lemma}
\begin{proof}
	Let $E\subset X$ be a $m$-null set such that whenever $x\in X\setminus E$ we have that 
	\begin{align*}
		e_x(t)=\lim_{r\to 0}\dashint_{B(x,r)}e_y(t)\ud m_J(y)
	\end{align*}
	for a countable dense set of $t$'s. (Recall that the measures $m_s$ are infinitesimally doubling, and therefore $m$ is a Vitali measure, cf. \cite[Theorem 3.4.3]{HKST07}). Estimate \eqref{eq:e_B-concavity} yields
	\begin{align*}
		\frac 1{m(B(x,r))}e^{-\frac 1N\dashint_{B(x,r)}e_y(t)\ud m}\ge \frac{\sigma_{K/N}^{1-t}(D)}{m(B(x,r))}e^{-\frac 1N\dashint_{B(x,r)}e_y(t)\ud m}+\frac{\sigma_{K/N}^{t}(D)}{m(B(x,r))}e^{-\frac 1N\dashint_{B(x,r)}e_y(t)\ud m}
	\end{align*}
for all $t\in [0,1]$ and $r>0$, cf. \eqref{eq:entropy}. Multiplying by $m(B(x,r))$ and letting $r\to 0$ we obtain the claimed inequality for a countable dense set of $t$'s. Since, for every $x\in X\setminus E$, $e_x$ (as well as $s\mapsto \sigma_{K/N}^s(D)$) is continuous, the inequality follows for all $t\in [0,1]$, completing the proof of the claim.
\end{proof}

\begin{proof}[Proof of Theorem \ref{thm:k,n-convex}]
Let $a,b\in {\rm int} J$ and $\tilde a_t\defeq H(a)+t(H(b)-H(a))$ for $t\in [0,1]$. For small $\delta>0$, the curve in $t\mapsto \tilde \nu_t^\delta$ in $\mathcal P([0,L])$ given by
\begin{align*}
	\tilde \nu_t^\delta\defeq \frac{\chi_{[\tilde a_t-\delta,\tilde a_t+\delta]}}{2\delta}\mathcal L^1|_{[0,L]}
\end{align*}
is a $W_2$-geodesic, and consequently $\nu_t^\delta\defeq h\inv_\ast\tilde\nu_t^\delta$ is a geodesic in $\mathcal P(J)$ with respect to $W_f$. We have $W_f(\nu_0^\delta,\nu_1^\delta)=W_2(\tilde\nu_0^\delta,\tilde\nu_1^\delta)=|H(b)-H(a)|=|b-a|_h$. Observe that
\begin{align*}
\nu_t^\delta=\frac{\chi_{B_{|\cdot|_h}(a_t,\delta)}}{2\delta}\mathcal H^1_{|\cdot|_h}=\frac{\chi_{B_{|\cdot|_h}(a_t,\delta)}}{2\delta g(\cdot,x)}\lambda_x,
\end{align*}
where $a_t\defeq H\inv(\tilde a_t)$ is the $|\cdot|_h$-geodesic $a\curvearrowright b$. Thus 
\begin{align*}
	e_x^\delta(t)\defeq {\rm Ent}_{\lambda_x}(\nu_t^\delta)=\int_{B_{|\cdot|_h}(a_t,\delta)}\frac{1}{2\delta g(\cdot,x)}\log\frac{1}{2\delta g(\cdot,x)}\ud\lambda_x=-\dashint_{B_{|\cdot|_h}(a_t,\delta)}\log g(\cdot,x)\ud\mathcal H^1_{|\cdot|_h}+\log\frac{1}{2\delta}.
\end{align*}
By Lemma \ref{lem:e_x-KN-convex} we may find a $m$-null set $E\subset X$ such that, if $x\in X\setminus E$, $e_x^\delta$ satisfies \eqref{eq:KN-convexity of e_x} for all $a,b\in {\rm int}J\cap \Q$ with $a<b$ and all $\delta\in \Q_+$.

For $\delta\in \Q_+$ and $x\in X\setminus E$, define
\begin{align}\label{eq:good-rep-of-g_J}
u_x^\delta(a)\defeq \exp\left(\frac 1N\dashint_{B_{|\cdot|_h}(a,\delta)}\log g(\cdot,x)\ud\mathcal H^1_{|\cdot|_h} \right),\quad a\in J_\delta\defeq (\inf J+\delta,\sup J-\delta).
\end{align}
For each $a,b\in J_\delta\cap \Q$ with $a<b$ we have that $u_x^\delta(a_t)\defeq 2\delta e^{-e_x^\delta(t)/N}$ satisfies \eqref{eq:KN-convexity of e_x} and by continuity the same holds true for all $a,b\in J_\delta$ with $a<b$. It follows that, for $\delta\in \Q_+$ and $x\in X\setminus E$, $u_x^\delta$ satisfies
\begin{align*}
	u_x^\delta(a_t)\ge \sigma_{K/N}^{1-t}(|b-a|_h)u_x^\delta(a)+\sigma_{K/N}^{t}(|b-a|_h)u_x^\delta(b)
\end{align*}
whenever $a,b\in J_\delta$, $t\in [0,1]$. 

Now on the one hand, for $m$-a.e. $x\in X\setminus E$, we have 
\begin{align*}
	\lim_{\delta\to 0^+}u_x^\delta(a)= g_J(a,x)^{1/N}\quad a.e.\ a\in J,
\end{align*}
while on the other hand, up to a subsequence, $u_x^\delta$ converges locally uniformly as $\delta\to 0^+$ to a continuous function $u_x:J\to \R$ satisfying
\begin{align*}
	u_x(a_t)\ge \sigma_{K/N}^{1-t}(|b-a|_h)u_x(a)+\sigma_{K/N}^{t}(|b-a|_h)u_x(b),\quad a,b\in J,\ t\in [0,1],
\end{align*}
whenever $x\in X\setminus E$. Thus for $m$-a.e. $x\in X$ we have that $u_x^N$ is a $\mathcal H^1_{|\cdot|_h}$-representative of $g(\cdot,x)$ and satisfies \eqref{eq:log-KN-convex}.
\end{proof}

\begin{proof}[Proof of Corollary \ref{cor:partial-rigidity}]
Suppose that $I=\R$, $h\equiv 1$, and $K=0$. Let $J_k=[-k,k]$. By Theorem \ref{thm:k,n-convex} there is a set $E\subset X$ with $m(E)=0$, and $(J_k,|\cdot|, g_{x}\mathcal H^1_{|\cdot|_h})$ is an RCD($0,N$)-space and $ g_{x}^{1/N}$ is concave in $J_k$ for all $k$, whenever $x\notin E$. Thus $g_x^{1/N}$ is a non-negative concave function on $\R$, and therefore it must be constant, which we denote by $c(x)$. It follows that 
\begin{align*}
	\int_Jm_s(E)\ud s=\bm m(J\times E)=\int_{J\times E}g(s,x)\ud s\ud m(x)=|J|\int_Ec\,\ud m
\end{align*}
for any compact $J\subset I$ and Borel $E\subset X$. The claim readily follows from this.
\end{proof}

\appendix

\section{Caratheodory theory of ODE's}\label{sec:ODE}
Let $D\subset \R^2$ be open. We say that $\Phi:D\to \R$ satisfies the \emph{Carath\'eodory condition} if $\Phi(\cdot,t)$ is continuous for fixed $t$ and $f(s,\cdot)$ is measurable for fixed $s$. We say that $\Phi$ satisfies a (local) \emph{Carath\'eodory-Lipschitz condition} if in addition, for every compact $I\times J\subset D$ with $I,J$ closed intervals, there exists $L\in L^1(J)$ such that
\begin{align}\label{eq:car-lip}
	| \Phi(s,t)- \Phi(s',t)|\le L(t)|s-s'|,\quad s,s'\in I.
\end{align}

We have the following existence and uniqueness theorem.
\begin{thm}[Theorem XVIII and XX, Ch. III §10 in \cite{wal98}]\label{thm:car-ODE}
If $\Phi$ is locally Carath\'eodory-Lipschitz and $(y_0,a)\in D$, then the ODE
\begin{equation}\label{eq:ODE_caratheodory}
y'=\Phi(y,t)
\end{equation}
with initial value $y(a)=y_0$ has a unique solution which can be extended to the left and right (of $a$) up to the boundary of $D$. In particular the unique solution varies continuously with $y_0$. 
\end{thm}

The following is a differential comparison theorem.
\begin{thm}[Theorem XXI, Ch. III §10 in \cite{wal98}]\label{thm:car-comparison}
Suppose $\Phi$ is locally Carath\'eodory-Lipschitz continuous, and that $\varphi,\psi\in AC(J)$, $J=[\xi,\xi+a]$ satisfy
\begin{itemize}
	\item[(a)] $\varphi(\xi)\le \psi(\xi)$,
	\item[(b)] $P\varphi\le P\psi$ a.e. in $J$, where $P\phi=\phi'-\Phi(\phi,x)$. 
\end{itemize}
Then either $\varphi<\psi$ in $J$ or there exists $c\in [\xi,\xi+a]$ s.t. $\varphi=\psi$ on $[\xi,c]$ and $\varphi<\psi$ in $(c,\xi+a]$. A corresponding statement holds for the interval $J'=[\xi-a,\xi]$ with the inequality reversed in (b).
\end{thm}

\bibliographystyle{plain}
\bibliography{abib}
\end{document}